\documentclass[a4paper,10pt]{amsart}
\usepackage{graphicx}
\usepackage{amsmath,amssymb,xy,array}
\usepackage[T1]{fontenc}
\usepackage{amsfonts}
\usepackage{amsmath}
\usepackage{amssymb}
\usepackage{amsthm}
\usepackage{xcolor}
\usepackage[utf8]{inputenc}
\usepackage{hyperref}
\usepackage{fancyhdr}
\usepackage{tikz}
\usepackage{tikz-cd}
\usepackage{longtable}
\usepackage{geometry}
\usepackage{textcomp}
\usetikzlibrary{trees}
\usetikzlibrary{arrows}
\usepackage{multirow}
\usepackage{youngtab}
\usepackage{mathdots}

\input{xy}
\xyoption{all}
\definecolor{gray0}{gray}{0.40}
\definecolor{gray1}{gray}{0.60}
\definecolor{gray2}{gray}{0.70}

\newcommand{\C}{\mathbb C}
\newcommand{\G}{\mathbb G}
\renewcommand{\P}{\mathbb P}
\newcommand{\R}{\mathbb R}

\newcommand{\Z}{\mathbb Z}
\DeclareMathOperator{\codim}{codim}
\newcommand{\rddots}{\reflectbox{$\ddots$}}

\newcommand{\Aut}{\operatorname{Aut}}
\DeclareMathOperator{\Bl}{Bl}

\DeclareMathOperator{\NE}{NE}

\DeclareMathOperator{\mult}{mult}

\DeclareMathOperator{\Exc}{Exc}

\DeclareMathOperator{\Eff}{Eff}
\DeclareMathOperator{\Nef}{Nef}
\DeclareMathOperator{\Mov}{Mov}
\DeclareMathOperator{\Pic}{Pic}

\DeclareMathOperator{\mov}{mov}

\renewcommand{\P}{\mathbb{P}}

\newtheorem{thm}{Theorem}[section]

\newtheorem{Question}[thm]{Question}
\newtheorem{Lemma}[thm]{Lemma}
\newtheorem{Proposition}[thm]{Proposition}

\newtheorem{Corollary}[thm]{Corollary}

\theoremstyle{definition}

\newtheorem{Definition}[thm]{Definition}
\newtheorem{Remark}[thm]{Remark}

\newtheorem{Example}[thm]{Example}
\newtheorem{Notation}[thm]{Notation}

\hypersetup{pdfpagemode=UseNone}
\hypersetup{pdfstartview=FitH}

\begin{document}

\title{Spherical blow-ups of Grassmannians and Mori Dream Spaces}

\author[Alex Massarenti]{Alex Massarenti}
\address{\sc Alex Massarenti\\
Universidade Federal Fluminense (UFF)\\
Rua M\'ario Santos Braga\\
24020-140, Niter\'oi, Rio de Janeiro\\ 
Brazil}
\email{alexmassarenti@id.uff.br}

\author[Rick Rischter]{Rick Rischter}
\address{\sc Rick Rischter\\
Universidade Federal de Itajub\'a (UNIFEI)\\ 
Av. BPS 1303, Bairro Pinheirinho\\ 
37500-903, Itajub\'a, Minas Gerais\\ 
Brazil}
\email{rischter@unifei.edu.br}

\date{\today}

\begin{abstract}
In this paper we classify weak Fano varieties that can be obtained by blowing-up general points in prime Fano varieties. We also classify spherical blow-ups of Grassmannians in general points, and we compute their effective cone. These blow-ups are, in particular, Mori dream spaces. Furthermore, we compute the stable base locus decomposition of the blow-up of a Grassmannian in one point, and we show how it is determined by linear systems of hyperplanes containing the osculating spaces of the Grassmannian at the blown-up point, and by the rational normal curves in the Grassmannian passing through the blown-up point.
\end{abstract}

\maketitle 

\setcounter{tocdepth}{1}

\tableofcontents

\section{Introduction}
\textit{Mori dream spaces} were introduced by Y. Hu and S. Keel in \cite{HK00}. The birational geometry of a Mori dream space $X$ can be encoded in some finite data, namely its cone of effective divisors $\Eff(X)$ together with a chamber decomposition on it, called the \textit{Mori chamber decomposition} of $\Eff(X)$. We refer to Section \ref{sec1} and the references therein for the rigorous definition and special properties of Mori dream spaces.

Mori dream spaces can be algebraically characterized as varieties whose total coordinate ring, called the \textit{Cox ring}, is finitely generated. Cox rings of projective varieties have been studied in various contexts \cite{AHL10}, \cite{AL11}, \cite{AGL16}, \cite{CT06}, \cite{DHHKL15}, \cite{HKL16}, \cite{Mu01}. 

In addition to this algebraic characterization there are several algebraic varieties characterized by positivity properties of the anti-canonical divisor that turn out to be Mori dream spaces \cite[Corollary 1.3.2]{BCHM10}.

Mori dream spaces obtained by blowing-up points in projective spaces have been studied in a series of papers \cite{CT06}, \cite{Mu01}, \cite{AM16}, and their relationships with moduli spaces of pointed rational curves and of parabolic vector bundles on $\mathbb{P}^1$ have been investigated as well \cite{Ba91}, \cite{AC16}, \cite{BM17}. In this paper we consider more general prime Fano varieties.

Recall that a smooth non degenerate projective variety $X\subset \P^N$ is a \textit{prime Fano variety} of index $\iota_X$ if its Picard group is generated by the class $H$ of a hyperplane section and $-K_{X}=\iota_{X}H$ for some positive integer $\iota_{X}$. By the work of Mori \cite{Mo79} if $\iota_X > \frac{\dim(X)+1}{2}$ then $X$ is covered by lines.

\begin{Notation}
Given a variety $X\subset\mathbb{P}^N$ we denote by $\pi:X_k\rightarrow X$ its blow-up at $k$ general points $p_1,\dots,p_k\in X$, by $E_1,\dots,E_k$ the exceptional divisors, and by $e_i$ the class of a general line contained in $E_i,$ for $i=1,\dots,k$. Therefore, $\Pic(X_k)=\pi^*(\Pic(X))\oplus\Z[E_1]\oplus \cdots
\oplus \Z[E_k]$, and similarly the space of curves $N_1(X_k)$ is generated by the strict transform of curves in $X$ and by $e_1,\dots,e_k$. When $k=1$ we use $E$ and $e$ instead of $E_1$ and $e_1$.

Furthermore, we denote by $H$ the class of a general hyperplane section in $X$, and also the pull-back in $\Pic(X_k)$ of a general hyperplane section. Similarly, we denote by $h$ both the class of a general line in $X$ and its strict transform in $X_k$.
\end{Notation}

In Section \ref{wFpF}, under suitable hypothesis, we manage to describe the Mori cone of effective curves in $X_k$, and this leads us to a classification of weak Fano varieties that can be obtained by blowing-up general points in a prime Fano variety. The following theorem should be viewed as a generalization of the analogous result for projective spaces proved in \cite[Proposition 2.9]{BL12}.  

\begin{thm}\label{main1}
Let $X\subset\mathbb{P}^N$ be an $n$-dimensional prime Fano variety covered by lines, and let $X_k$ be the blow-up of $X$ at $k\geq 1$ general points. Then $X_k$ is Fano if and only if
\begin{itemize}
\item[-] $X\cong \P^n$ and either $n=2$ and $k\leq 8$, or $n\geq 3$ and $k = 1$;
\item[-] $X\cong Q^n\subset\mathbb{P}^{n+1}$ is a smooth quadric and either $n=2$ and $k\leq 7$, or $n\geq 3$ and $k\leq 2$. 
\end{itemize}
Furthermore, $X_k$ is weak Fano but not Fano if and only if 
\begin{itemize}
\item[-] $X\cong \P^3$ and $2 \leq k\leq 7$;
\item[-] $X\cong Q^3\subset\mathbb{P}^{4}$ is a smooth quadric and $3\leq k\leq 6$;
\item[-] $X\cong Y_3\subset\mathbb{P}^{n+1}$, with $n\geq 3$, is a smooth cubic hypersurface and $k\leq 2$;
\item[-] $X\cong Y_{2,2}\subset\mathbb{P}^{n+2}$, with $n\geq 3$, is a smooth complete intersection of two quadrics and $k\leq 3$;
\item[-] $X$ is a linear section of codimension $c\leq 3$ of $\mathbb{G}(1,4)\subset\mathbb{P}^9$, the Grassmannian of lines in $\mathbb{P}^4$, and $k\leq 4$.
\end{itemize}
\end{thm}
Most of the Fano varieties appearing in Theorem \ref{main1} are well-known. For instance, since the blow-up of $Q^2$ at $k$ general points is isomorphic to the blow-up of $\mathbb{P}^2$ at $k+1$ general points the cases $X \cong \mathbb{P}^2$ and $X\cong Q^2$ follow from the classical construction of del Pezzo surfaces as blow-ups of $\mathbb{P}^2$. We did not find the case $X\cong Q^n$ with $n\geq 3$ and $k\leq 2$ in the literature but we believe it is well-known to the experts. At the best of our knowledge, with the exception of the case $X\cong\mathbb{P}^3$ which has been treated in \cite[Proposition 2.9]{BL12}, all the other varieties in Theorem \ref{main1} are new examples of weak Fano varieties.   

In particular, Theorem \ref{main1} provides new examples of Mori dream spaces. We would like to mention that the nef cones of some of the blow-ups appearing in Theorem \ref{main1} have been computed in \cite[Theorem 4.1]{LTU15}.

Next we turn our attention to another class of projective varieties, namely spherical varieties. Given a reductive algebraic group $G$ and a Borel subgroup $B$, a \textit{spherical variety} is a $G$-variety with an open dense $B$-orbit. For instance flag varieties and toric varieties are spherical varieties. 
Spherical varieties are Mori dream spaces. This follows from the work of M. Brion \cite{Br93}, \cite{Br07}, and a more explicit proof can be found in \cite[Section 4]{Pe14}.

In Sections \ref{sGbu} and \ref{nsGbu} we concentrate on a particular prime Fano variety, namely the Grassmannian $\G(r,n)\subset\mathbb{P}^N$, parametrizing $r$-planes in $\mathbb{P}^n$, under its Pl\"ucker embedding. Exploiting the spherical nature of $\G(r,n)$ we manage to classify the blow-ups $\G(r,n)_k$ that are spherical. Furthermore, we compute their effective cones. 

Let us recall that the pseudo-effective cone $\overline{\Eff}(X)$ of a projective variety $X$ with $h^1(X,\mathcal{O}_X)=0$ can be decomposed into chambers depending on the stable
base locus of the corresponding linear series, see Section \ref{mcdsbl} for details. Such decomposition called \textit{stable base locus decomposition} in general is coarser than the Mori chamber decomposition. In the case $k=1$ we determine the stable base locus decomposition of $\Eff(\G(r,n)_1)$.

\begin{thm}\label{main2}
Let $\G(r,n)_k$ be the blow-up of the Grassmannian $\G(r,n)$ at $k$ general points, with $n\geq 2r+1$. The variety $\G(r,n)_k$ is spherical if and only if 
\begin{itemize}
\item[-] either $r = 0$ and $k\leq n+1$; or
\item[-] $k = 1$ and $r,n$ are arbitrary; or
\item[-] $k = 2$ and $r =1$, or $n\in\{2r+1,2r+2\}$; or
\item[-] $k = 3$ and $(r,n) = (1,5)$.
\end{itemize}
Furthermore, for the effective cones of the spherical blow-ups we have
\begin{itemize}
\item[-] $\Eff(\G(r,n)_1) = \left\langle E, H-(r+1)E\right\rangle$;
\item[-] $\Eff(\G(1,n)_2) = \left\langle E_1, E_2, H-2(E_1+E_2)\right\rangle$;
\item[-] $\Eff(\G(r,2r+1)_2) = \left\langle E_1, E_2, H-(r+1)E_1, H-(r+1)E_2\right\rangle$;
\item[-] $\Eff(\G(r,2r+2)_2) = \left\langle E_1, E_2, H-(r+1)E_1-E_2,H-E_1-(r+1)E_2\right\rangle$;
\item[-] $\Eff(\G(1,5)_3) = \left\langle E_1, E_2,E_3, H-2(E_1+E_2), H-2(E_1+E_3), H-2(E_2+E_3)\right\rangle$.
\end{itemize}
Moreover, the movable cone of $\G(r,n)_1$ is given by
$$\Mov(\G(r,n)_1)=\begin{cases}
\left\langle H,H-rE\right\rangle &\mbox{ if } n=2r+1\\
\left\langle H,H-(r+1)E\right\rangle &\mbox{ if } n>2r+1
\end{cases}.$$

The Mori cone of $\G(r,n)_1$ is given by $\NE(\G(r,n)_1)=\left\langle e,h-e\right\rangle$, while its cone of moving curve is $\mov(\G(r,n)_1) = \left\langle h,(r+1)h-e \right\rangle$. 

Finally, $\Nef(\G(r,n)_1)  =\left\langle H, H-E\right\rangle$ and the divisors $E,H,H-E,\dots,H-(r+1)E$ give the walls of the stable base locus decomposition of $\Eff(\G(r,n)_1)$ as represented in the following picture
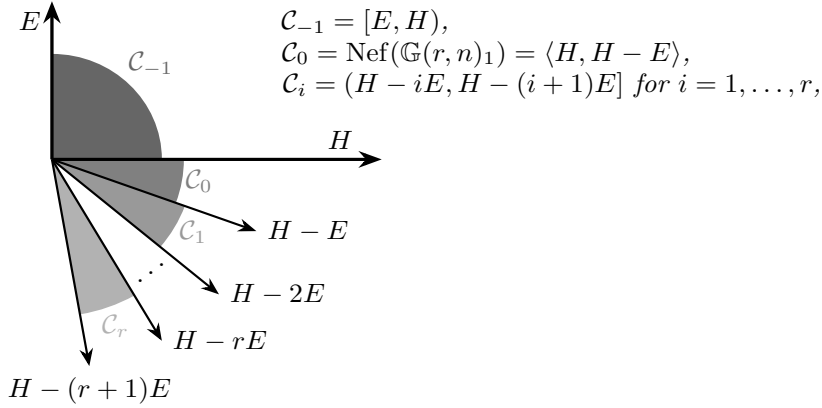
\vspace{-0.2cm}
\begin{figure}[htb!]
\centering			
\resizebox{0.72\textwidth}{0.36\textwidth}{%
\begin{tikzpicture}[>=Stealth,scale=1.4]
\fill[gray0] (0,0) -- (1.0,0)
arc [start angle=0, end angle=90,
radius=1.0];
\fill[gray] (0,0) -- (1.2,0)
arc [start angle=0, end angle=-20, radius=1.2];
\fill[gray1] (0,0) -- (0.93*1.3,-0.34*1.3)
arc [start angle=-20, end angle=-40, radius=1.3];
\fill[gray2] (0,0) -- (0.5*1.5,-0.86*1.5)
arc [start angle=-60, end angle=-80, radius=1.5];
\draw[->][very thick] (0,0) -- (0,1.5)   node[left,very near end]{$E$};
\draw[->][very thick] (0,0) -- (3,0)   node[above,very near end]{$H$};
\draw (0.9,0.9) node[very thick,gray0]{$\mathcal{C}_{-1}$};
%I used as angles 20,40,60, and 80 degrees.
%The crazy numbers between 0 and 1 below are sine and cossine of 20,40,60, and 80.
\draw[->][thick] (0,0) -- (0.93*2,-0.34*2) 
node[right,very thick]{$H-E$}
node[above right, pos=0.6, very thick,gray]{$\mathcal C_0$}; 
\draw[->][thick] (0,0) -- (0.76*2,-0.64*2) 
node[right,very thick]{$H-2E$}
node[above right, pos=0.7, very thick,gray1]{$\mathcal C_1$}; \draw (0.9,-1) node[very thick]{$\rddots$};
\draw[->][thick] (0,0) -- (0.5*2,-0.86*2)
node[right,very thick]{$H-rE$}
node[below left, pos=0.8, very thick,gray2]{$\mathcal C_r$};  
\draw[->][thick] (0,0) -- (0.17*2,-0.98*2)
node[below,very thick]{$H-(r+1)E$};
\draw (2.85,1.3) node[very thick]{$\mathcal{C}_{-1} =  [E,H)$,};
\draw (3.95,1) node[very thick]{$\mathcal{C}_0 = \Nef(\G(r,n)_1) = \left\langle H,H-E\right\rangle$,};
\draw (4.55,0.7) node[very thick]{$\mathcal{C}_i = (H-iE,H-(i+1)E]$ for $i = 1,\dots,r$,};
\end{tikzpicture}}
\end{figure}%
\vspace{-0.2cm}
where with the notation $\mathcal{C}_i = (H-iE,H-(i+1)E]$ we mean that the ray spanned by $H-(i+1)E$ belongs to $\mathcal{C}_i$ but the ray spanned by $H-iE$ does not, and similarly with the notation $\mathcal{C}_{-1} = [E,H)$ we mean that the ray spanned by $E$ belongs to $\mathcal{C}_{-1}$ but the ray spanned by $H$ does not.
\end{thm}

Note that in the toric case $r = 0$ we recover the well known fact that $\Eff(\mathbb{P}^n_1) = \left\langle E,H-E\right\rangle$ with stable base locus decomposition given by $E,H,H-E$.

Furthermore, we show how the stable base locus decomposition of $\Eff(\G(r,n)_1)$ can be described in terms of linear systems of hyperplanes containing the osculating spaces of $\G(r,n)$ at the blown-up point. Recall that given a smooth point $p\in X\subset\mathbb{P}^N$, the \textit{$m$-osculating space} $T_p^{m}X$ of $X$ at $p$ is roughly speaking the smallest linear subspace locally approximating $X$ up to order $m$ at $x$, see Definition \ref{oscdef}. 

While doing this we show in Lemma \ref{lemmaBS} a result of independent interest.
We interpret the intersection $R_m:=\G(r,n)\cap T^m_p\G(r,n)$ as a Schubert variety and we show that $R_m$ is the subvariety of $\G(r,n)$ swept out by the degree $m$ rational normal curves in $\G(r,n)$ passing through $p\in\G(r,n)$. This generalizes the well-known fact that the tangent space $T_p\G(r,n)$ intersects $\G(r,n)$ exactly in the cone of lines in $\G(r,n)$ with vertex $p \in \G(r,n)$.

We would like to mention that J. Kopper, while describing the effective cycles on blow-ups of Grassmannians, independently computed the effective cone of $\G(r,n)_1$ and $\G(r,n)_2$ in \cite{Ko16}. 

Finally, we have the following result on Mori dream spaces obtained by blowing-up general points in a Grassmannian.

\begin{thm}\label{main3}
Let $\G(r,n)_k$ be the blow-up of the Grassmannian $\G(r,n)$ at $k$ general points. If one of the following occurs:
\begin{itemize}
\item[-] $r=0$ and $k\leq n+3$; 
\item[-] $k = 1$ and $r,n$ are arbitrary; 
\item[-] $k=2$, and one of the following conditions is satisfied:
\begin{itemize}
\item[-] $r = 1$ and $n$ is arbitrary,
\item[-] $r = 2$ and $n$ is arbitrary,
\item[-] $n = 2r+1$ and $r$ is arbitrary,
\item[-] $n = 2r+2$ and $r$ is arbitrary,
\end{itemize}
\item[-] $k=3$, and $r=1, n\geq 4$, or  $(r,n)=(2,8)$;
\item[-] $k=4$, and $(r,n)\in\{(1,4),(1,7)\}$,
\end{itemize}
then $\G(r,n)_k$ is a Mori dream space.
\end{thm}

It would be interesting to have a classification of Mori dream spaces that can be obtained by blowing-up points in Grassmannians in the spirit of the main results for blow-ups of projective spaces in \cite{Mu01,CT06,AM16}. In this direction we would like to mention that J. Kopper proved in \cite[Theorem 6.5]{Ko16} that if the Segre–Harbourne–Gimigliano–Hirschowitz conjecture holds for ten very general points in $\mathbb{P}^2$, then $\Eff(\G(1,4)_6)$ is not finitely generated. Therefore, in particular $\G(1,4)_k$ is not a Mori dream space for $k\geq 6$.

\subsection*{Organization of the paper}
This paper is organized as follows. In Section \ref{sec1}, we introduce spherical varieties, weak Fano varieties, Mori dream spaces and collect some of their basic properties. In Section \ref{wFpF} we compute the Mori cones of effective curves of blow-ups of prime Fano varieties at general points and we prove Theorem \ref{main1}. In Sections \ref{sGbu} and \ref{nsGbu} we classify spherical blow-ups of Grassmannians at general points, compute their effective cones, and prove Theorem \ref{main3}. Finally, in Section \ref{mcdsbl} we compute the stable base locus decomposition of the effective cone of the blow-up of $\G(r,n)$ at a point completing the proof of Theorem \ref{main2}.

\subsection*{Acknowledgments}
We thank Carolina Araujo for useful discussion on the stable base locus decomposition of blow-ups of Grassmannians, and Antonio Laface for helpful comments on a preliminary version of the paper. Finally, we thank the referee for pointing out a gap in the first version of the paper, and for the careful reading that helped us to improve the exposition.

The first named author is a member of the Gruppo Nazionale per le Strutture Algebriche, Geometriche e le loro Applicazioni of the Istituto Nazionale di Alta Matematica "F. Severi" (GNSAGA-INDAM). The second named author would like to thank CNPq for the financial support.

\section{Spherical varieties, Mori dream spaces and weak Fano varieties}\label{sec1}
Throughout the paper $X$ will be a normal projective variety over an algebraically closed field of characteristic zero. We denote by $N^1(X)$ the real vector space of $\mathbb{R}$-Cartier divisors modulo numerical equivalence. 
The \emph{nef cone} of $X$ is the closed convex cone $\Nef(X)\subset N^1(X)$ generated by classes of 
nef divisors. 
The \emph{movable cone} of $X$ is the convex cone $\Mov(X)\subset N^1(X)$ generated by classes of 
\emph{movable divisors}. These are Cartier divisors whose stable base locus has codimension at least two in $X$.
The \emph{effective cone} of $X$ is the convex cone $\Eff(X)\subset N^1(X)$ generated by classes of 
\emph{effective divisors}. We have inclusions $\Nef(X)\ \subset \ \overline{\Mov(X)}\ \subset \ \overline{\Eff(X)}$. We refer to \cite[Chapter 1]{De01} for a comprehensive treatment of these topics. 

\subsection{Spherical varieties}
Recall that an algebraic group $G$ is solvable when it is solvable as an abstract group. A Borel subgroup $B$ of an algebraic group $G$ is a subgroup which is maximal among the connected solvable algebraic subgroups of $G$. The radical $R(G)$ of an algebraic group is the identity component of the intersection
of all Borel subgroups of $G$. We say that $G$ is semi-simple if $R(G)$ is trivial. We say that $G$ is reductive if the unipotent part of $R(G)$, i.e. the subgroup of unipotent elements of $R(G)$, is trivial.

Given an algebraic group $G$ there is a single conjugacy class of Borel subgroups. For instance, in the group $GL_n$ of $n\times n$ invertible matrices, the subgroup of invertible upper triangular matrices is a Borel subgroup. The radical of $GL_n$ is the subgroup of scalar matrices, therefore $GL_n$ is reductive but not semi-simple. On other hand, $SL_n$ is semi-simple.

\begin{Definition}
A \textit{spherical variety} is a normal variety $X$ together with an action of a connected reductive affine algebraic group $G$, a Borel subgroup $B\subset G$, and a base point $p_0\in X$ such that the $B$-orbit of $p_0$ in $X$ is a dense open subset of $X$. 

The \textit{complexity} $c(X)$ of a normal variety $X$ with an action of a connected reductive affine algebraic group $G$ is the minimal codimension in $X$ of an orbit of a Borel subgroup $B\subset G$. Therefore, a spherical variety is a normal $G$-variety of complexity zero.
\end{Definition}

\begin{Notation}
Throughout the paper we will always view the Grassmannian $\G(r,n)\subset\mathbb{P}^N$ of $r$-planes in $\mathbb{P}^n$, with $N = \binom{n+1}{r+1}-1$, as a projective variety in its Pl\"ucker embedding.
\end{Notation}

\begin{Example}\label{G(r,n) spherical}
Any toric variety is a spherical variety with $B=G$ equal to the torus. Consider $X:=\G(r,n), \ G:=SL_{n+1}$. Choose a complete flag $\{0\}= V_0\subset V_1\subset \cdots \subset V_{n+1}=\C^{n+1}$
of linear spaces in $\C^{n+1}$, with $V_{r+1}$ corresponding to a point $p\in \G(r,n).$ 
Let $B$ be the only Borel subgroup of $G$ that stabilizes this flag, and choose a basis 
$e_0,\dots, e_n$ of $\C^{n+1}$ such that $B$ is the subgroup of upper triangular matrices in this basis.
Consider the divisor $D=(p_{n-r,n-r+1,\dots,n}=0)$ and the point $p_0=[0:\dots :0:1]\in \G(r,n)\backslash D$. We have that $B\cdot p_0=\G(r,n)\backslash D$, and hence $(X,G,B,p_0)$ is a spherical variety.
\end{Example}

Next, we recall that the effective cone of a spherical variety can be described in terms of divisors which are invariant under the action of the Borel subgroup.

\begin{Definition}
Let $(X,G,B,p)$ be a spherical variety. We distinguish two types of $B$-invariant prime divisors:
\begin{itemize}
	\item[-] A \textit{boundary divisor} of $X$ is a $G$-invariant prime divisor on $X.$
	\item[-] A \textit{color} of $X$ is a $B$-invariant prime divisor that is not $G$-invariant.
\end{itemize}
\end{Definition}

For instance, for a toric variety there are no colors, and the boundary divisors are the usual toric invariant divisors. For a spherical variety we have to take into account the colors as well.

\begin{Proposition}\cite[Proposition 4.5.4.4]{ADHL15}\label{effcone spherical}
Let $(X,G,B,p_0)$ be a spherical variety.
\begin{itemize}
	\item[-] There are finitely many boundary divisors $E_1,\dots,E_r$ and finitely many colors $D_1,\dots,D_s$ on $X$. Furthermore, $X\backslash \ B\cdot p_0=E_1\cup\dots\cup E_r \cup D_1\cup\dots\cup D_s$.
	\item[-] The classes of the $E_k$'s and of the $D_i$'s generate $\Eff(X)\subset N^1(X)$ as a cone.
\end{itemize}
\end{Proposition}

\subsection{Mori dream spaces and weak Fano varieties}
Let $X$ be a normal $\mathbb{Q}$-factorial variety. We say that a birational map  $f: X \dasharrow X'$ to a normal projective variety $X'$  is a \emph{birational contraction} if its
inverse does not contract any divisor. 
We say that it is a \emph{small $\mathbb{Q}$-factorial modification} 
if $X'$ is $\mathbb{Q}$-factorial  and $f$ is an isomorphism in codimension one.
If  $f: X \dasharrow X'$ is a small $\mathbb{Q}$-factorial modification, then 
the natural pullback map $f^*:N^1(X')\to N^1(X)$ sends $\Mov(X')$ and $\Eff(X')$
isomorphically onto $\Mov(X)$ and $\Eff(X)$, respectively.
In particular, we have $f^*(\Nef(X'))\subset \overline{\Mov(X)}$.

\begin{Definition}\label{def:MDS} 
A normal projective $\mathbb{Q}$-factorial variety $X$ is called a \emph{Mori dream space}
if the following conditions hold:
\begin{enumerate}
\item[-] $\Pic{(X)}$ is finitely generated, or equivalently $h^1(X,\mathcal{O}_X)=0$,
\item[-] $\Nef{(X)}$ is generated by the classes of finitely many semi-ample divisors,
\item[-] there is a finite collection of small $\mathbb{Q}$-factorial modifications
 $f_i: X \dasharrow X_i$, such that each $X_i$ satisfies the second condition above, and $
 \Mov{(X)} \ = \ \bigcup_i \  f_i^*(\Nef{(X_i)})$.
\end{enumerate}
\end{Definition}

The collection of all faces of all cones $f_i^*(\Nef{(X_i)})$ above forms a fan which is supported on $\Mov(X)$.
If two maximal cones of this fan, say $f_i^*(\Nef{(X_i)})$ and $f_j^*(\Nef{(X_j)})$, meet along a facet,
then there exist a normal projective variety $Y$, a small modification $\varphi:X_i\dasharrow X_j$, and $h_i:X_i\rightarrow Y$ and $h_j:X_j\rightarrow Y$ small birational morphisms of relative Picard number one such that $h_j\circ\varphi = h_i$. The fan structure on $\Mov(X)$ can be extended to a fan supported on $\Eff(X)$ as follows. 

\begin{Definition}\label{MCD}
Let $X$ be a Mori dream space.
We describe a fan structure on the effective cone $\Eff(X)$, called the \emph{Mori chamber decomposition}.
We refer to \cite[Proposition 1.11]{HK00} and \cite[Section 2.2]{Ok16} for details.
There are finitely many birational contractions from $X$ to Mori dream spaces, denoted by $g_i:X\dasharrow Y_i$.
The set $\Exc(g_i)$ of exceptional prime divisors of $g_i$ has cardinality $\rho(X/Y_i)=\rho(X)-\rho(Y_i)$.
The maximal cones $\mathcal{C}$ of the Mori chamber decomposition of $\Eff(X)$ are of the form: $\mathcal{C}_i \ = \left\langle g_i^*\big(\Nef(Y_i)\big) , \Exc(g_i) \right\rangle$. We call $\mathcal{C}_i$ or its interior $\mathcal{C}_i^{^\circ}$ a \emph{maximal chamber} of $\Eff(X)$.
\end{Definition}

There are varieties, characterized by positivity properties of the anti-canonical divisor, that turn out to be Mori dream spaces.

\begin{Definition}
Let $X$ be a smooth projective variety and $-K_X$ its anti-canonical divisor. We say that $X$ is \textit{Fano} if $-K_X$ is ample, and that $X$ is \textit{weak Fano} if $-K_X$ is nef and big.
\end{Definition}

By \cite[Corollary 1.3.2]{BCHM10} we have that weak Fano varieties are Mori dream spaces. In fact, there is a larger class of varieties called log Fano varieties, which includes weak Fano varieties and which are Mori dream spaces as well.

\begin{Remark}\label{rem1}
Our interest in spherical varieties comes from the fact that $\mathbb{Q}$-factorial spherical varieties are Mori dream spaces. This follows from the work of M. Brion \cite{Br93}. An alternative proof of this result can be found in \cite[Section 4]{Pe14}. More generally, by \cite[Section 4]{Pe14} we have that any normal $\mathbb{Q}$-factorial projective $G$-variety $X$ of complexity $c(X)\leq 1$ is a Mori dream space as well.
\end{Remark}

\section{Weak Fano blow-ups of prime Fano varieties}\label{wFpF}
In this section we determine which a blow-ups of prime Fano varieties at general points are Fano or weak Fano.

\begin{Lemma}\label{conecurvesgeneral}
Let $X\subset \P^N$ be a normal projective non degenerate variety of dimension $n\geq 2$, covered by lines and such that $N_1(X)=\R [h]$. If $k\leq \codim(X)+1$, then the Mori cone of $X_k$ is generated by $e_i,l_i, i=1,\dots,k$, where $l_i=h-e_i$ is the class of the strict transform of a general line through the blown-up point $p_i$.
\end{Lemma}
\begin{proof}
Let $\widetilde{C}\sim dh-m_1e_1-\cdots-m_ke_k$, with $d,m_1,\dots,m_k\in \Z$, be the class of an irreducible curve in $X_k$. If $\widetilde{C}$ is contracted by $\pi$, then $\widetilde{C}\sim ae_i$ with $a>0$ for some $i$. If $\widetilde{C}$ is mapped onto a curve $C\subset X\subset\mathbb{P}^N$, we  must have $d\geq m_i=\mult_{p_i}C\geq 0$ for any $i$. If $k=1$ we may write $c=ml+(d-m)h$ with $d-m$, $m\geq 0$ and we are done. Now, assume that $k\geq 2$. Note that it is enough to prove that $m_1+\cdots+m_k\leq d$. In fact, we could then write $\widetilde{C}\sim m_1l_1+\cdots + m_kl_k+(d-m_1-\cdots -m_k)h$, keeping in mind that $h=l_i+e_i$ for any $i$.

Now, assume by contradiction that $m_1+\cdots+m_k>d$, and consider $\Pi=\left\langle p_1\dots p_k\right\rangle\cong\P^{k-1}\subset \P^N$. Therefore, $C$ intersects $\Pi$ in at least $m_1+\cdots+m_k>d=\deg (C)$ points counted with multiplicity, and then $C\subset \Pi$. Now we distinguish two cases:
\begin{itemize}
\item[-] if $k< \codim(X)+1$ then $\dim(\Pi)+\dim(X)< N$, and the generalized Trisecant lemma \cite[Proposition 2.6]{CC02} yields that $\Pi\cap X = \{p_1,\dots,p_k\}$,
\item[-] if $k = \codim(X)+1$ then $\dim(\Pi)+\dim(X) = N$. Furthermore, since $X\subset\P^N$ is non degenerate $\deg(X)\geq \codim_{\P^N}(X)+1 = k$ and a general $(k-1)$-plane $\Pi$ of $\P^N$ intersects $X$ in $\deg(X)\geq k$ points. Since any subset of cardinality $k$ of the set of these $\deg(X)$ points generate $\Pi$ we get that the general $(k-1)$-plane in $\mathbb{P}^N$ is generated by $k$ general points of $X$. Therefore a $(k-1)$-plane $\Pi$ generated by $k$ general points of $X$ is general among the $(k-1)$-planes of $\P^N$, and hence such a $\Pi$ intersects $X$ in $\deg(X)$ points. In particular $\dim(\Pi\cap X) = 0$.
\end{itemize} 
In both cases we get a contradiction with $C\subset\Pi\cap X$.  
\end{proof}

We would like to mention that J. Kopper independently proved Lemma \ref{conecurvesgeneral} in \cite[Proposition 4.5]{Ko16} when $X = \G(r,n)\subset\mathbb{P}^N$ is a Grassmannian. He also extended the result to $\G(r,n)_k$ for $k\leq \codim(\G(r,n))+2$, and determined the effective cone of $2$-dimensional cycles for $\G(r,n)_k$, when $k\leq \codim(\G(r,n))+1$. Furthermore, note that the hypothesis $k\leq \codim(X)+1$ in Lemma \ref{conecurvesgeneral}, as the next result shows, is indeed necessary.

\begin{Lemma}\label{conecurvesquadrics}
Let $Q^n_k$ be the blow-up of a smooth quadric $Q^n\subset \P^{n+1}$ at $k\geq 3$ general points. Denote by $l_i$ the class $h-e_i$ of the strict transform of a general line through $p_i$, and by $c_{ijl}$ the class $2h-e_i-e_j-e_l$ of the strict transform of the conic through $p_i,p_j,p_l$.

Assume that $k\leq (3n+2)/2$ if $n$ is even, and that $k\leq (3n+3)/2$ if $n$ is odd. Then the Mori cone of $Q_k^n$ is generated by $e_i,l_i$ for $i=1,\dots,k$, and by $c_{ijl}$ for $1\leq i<j<l\leq k$.
\end{Lemma}
\begin{proof}
Note that given three general points $p_i,p_j,p_l$ the plane generated by them
intersects the quadric $Q^n$ in a conic whose strict transform under 
the blow-up has class $c_{ijl}=2h-e_i-e_j-e_l$. Clearly, it is enough to prove the result
for $k= (3n+2)/2$ if $n$ is even, and for $k= (3n+3)/2$ if $n$ is odd. We use the same notation of Lemma \ref{conecurvesgeneral}. Let $C$ be an effective curve in $Q^n_k$. 

If $C\sim dh$ then we may write $C\sim dl_1+de_1$. Otherwise, if $C\sim dh-m_ie_i-\dots -m_ke_k$ where $m_k\geq \dots \geq m_i > 0=m_{i-1}=\dots=m_1$ for some $i\in\{2,\dots,k\}$, then we may write $C\sim (dh -e_1-\dots -e_{i-1}-m_ie_1-\dots -m_ke_k)+e_1+\dots e_{i-1}$. Therefore, we may assume that $C\sim dh-m_1e_1-\dots -m_ke_k$ with $d\geq m_k\geq\dots\geq m_1 >0$.  

If $n=2n'-1$ is odd write
\begin{align*}
\widetilde{C}&\sim \sum _{i=1}^{n'}\Big[
m_{3i-2} c_{3i-2,3i-1,3i} +
(m_{3i-1}-m_{3i-2}) l_{3i-1}+(m_{3i}-m_{3i-2}) l_{3i} \Big]+\\
&+\left(d-\sum_{i=1}^{n'}(m_{3i-1}+m_{3i})\right)h.
\end{align*}

To conclude it is enough to prove that $d-(\sum_{i=1}^{n'}(m_{3i-1}+m_{3i}))\geq 0$. Assume by contradiction that $\sum_{i=1}^{n'}(m_{3i-1}+m_{3i})>d$. Then $C$ is contained in $\Pi=\left\langle p_2,p_3,\cdots, p_{3n'-1},p_{3n'}\right\rangle\cong\P^{2n'-1}=\P^n$. Since the points are in general position, $p_1\notin \Pi$ and $m_1=0$. A contradiction.

Now, consider the case $n=2n'$ even, and write
\begin{align*}
\widetilde{C}&\sim \sum _{i=1}^{n'-1}\Big[m_{3i-2}
 c_{3i-2,3i-1,3i} +
(m_{3i-1}-m_{3i-2}) l_{3i-1}+(m_{3i}-m_{3i-2}) l_{3i} \Big]\\
&+m_{3n'-2}c_{3n'-2,3n'-1,3n'}+(m_{3n'-1}-m_{3n'-2})c_{3n'-1,3n',3n'+1}+\\
&+(m_{3n'}-m_{3n'-1})l_{3n'}+(m_{3n'+1}-m_{3n'-1}+m_{3n'-2})l_{3n'+1}+\\
&+\left(d-\sum_{i=0}^{n'-1}(m_{3i-1}+m_{3i})-(m_{3n'-2}+m_{3n'}+m_{3n'+1})\right)h.
\end{align*}
In this case it is enough to prove that $d-\sum_{i=1}^{n'-1}(m_{3i-1}+m_{3i})-(m_{3n'-2}+m_{3n'}+m_{3n'+1})\geq 0$. Assume by contradiction that $\sum_{i=1}^{n'-1}(m_{3i-1}+m_{3i})+(m_{3n'-2}+m_{3n'}+m_{3n'+1}))>d$. Then $C$ is contained in $\Pi = \left\langle p_2, p_3, \dots, p_{3n'-4}, p_{3n'-3}, p_{3n'-2}, p_{3n'},p_{3n'+1}\right\rangle\cong\P^{2n'}=\P^n$. Again, since the points are in general position we have $p_1\notin \Pi$ and $m_1=0$. A Contradiction.
\end{proof}

In what follows we determine when the blow-up of a smooth quadric and the blow-up of a Grassmannian at general points are Fano or weak Fano.  

\begin{Proposition}\label{Q_k^n weak Fano}
Let $Q^n_k$ be the blow-up of a smooth quadric $Q^n\subset \P^{n+1}$ at $k\geq 1$ general points. Then
\begin{itemize}
 \item[-]$Q_k^n$ is Fano if and only if either $k\leq 2$ or $n=2$ and $k\leq 7$.
 \item[-]$Q_k^n$ is weak Fano if and only if one of the following holds:
\begin{itemize}
\item[-] $n=2$ and $k\leq 7$;
\item[-] $n=3$ and $k\leq 6$;
\item[-] $n\geq 4$ and $k\leq 2$.
\end{itemize}
\end{itemize}
\end{Proposition}
\begin{proof}
For $n=2$ the result follows from the identification $\P^2_2\cong Q^2_1$ and from the classification of del Pezzo surfaces. Assume $n\geq 3$, since $-K_{Q_k^n}=nH-(n-1)(E_1+\cdots+E_k)$ we have
$$-K_{Q_k^n}\cdot e_i= n-1>0 \mbox{ and }
-K_{Q_k^n}\cdot l_i=n-(n-1)=1>0.$$
Then by Lemma \ref{conecurvesgeneral} $-K_{Q_k^n}$ is ample for $k\leq 2$, and $Q_1^n$, $Q_2^n$ are Fano.

Now, assume $k\geq 3$ and observe that
$$-K_{Q_k^n}\cdot c_{ijl}\!=\!(nH-(n-1)(E_1+\cdots+E_k))\cdot(2h-e_i-e_j-e_l)\!=\!2n-3(n-1)=3-n.$$
Hence $-K_{Q_k^n}$ is not nef and $Q_k^n$ is not weak Fano for $n\geq 4$ and $k\geq 3$.

Finally, assume $n=3$. Note that Lemma \ref{conecurvesquadrics} gives the Mori cone of $Q_k^3$ for $k\leq 6$, and we have that $-K_{Q_k^3}\cdot c_{ijl}=0$. Therefore $-K_{Q_k^3}$ is nef but not ample for $3\leq k\leq 6$. Note that $H^3=\deg(Q^3)=2$ and
$$   (-K_{Q_k^3})^3=(3H-2(E_1+\dots+E_k))^3=3^3\cdot 2-2^3\cdot k=
54-8k>0 \Leftrightarrow k\leq 6.  $$
Therefore, by \cite[Section 2.2]{La04} $Q_k^3$ is weak Fano if and only if $k\leq 6$.
\end{proof}

\begin{Lemma}\label{wFdegree}
Let $X\subset \P^N$ be a prime Fano variety of dimension $n\geq 2$ and degree $d\geq 2$. If $k\geq \frac{d\iota_{X}^n}{(n-1)^n}$ then $X_k$ is not weak Fano. Furthermore, if $X$ is covered by lines and $k\leq \codim(X)+1$ then $X_k$ is weak Fano if and only if $k< \frac{d\iota_{X}^n}{(n-1)^n}$ and $\iota_X\geq n-1$.
\end{Lemma}
\begin{proof}
The anti-canonical divisor of $X_k$ is given by
$$-K_{X_k}= \iota_{X}H-(n-1)(E_1+\dots +E_k).$$
Assume that $X_k$ is weak Fano. Then $-K_{X_k}$ is nef, and since it is also big \cite[Section 2.2]{La04} implies that $(-K_{X_k})^n = d\iota_X^n-k(n-1)^n > 0$.

Now, assume that $X$ is covered by lines and $k\leq \codim(X)+1$. Since $(-K_{X_k})\cdot e_i = n-1 >0$, and $(-K_{X_k})\cdot l_i = \iota_{X}-(n-1)$ Lemma \ref{conecurvesgeneral} yields that $-K_{X_k}$ is nef if and only if $\iota_X\geq n-1$. Furthermore, since $(-K_{X_k})^n = d\iota^n-k(n-1)^n,$ then by \cite[Section 2.2]{La04} $-K_{X_k}$ is nef and big if and only if $k< \frac{d\iota_{X}^n}{(n-1)^n}$ and $\iota_X\geq n-1.$
\end{proof}

Next, we consider the case of Grassmannians. 

\begin{Proposition}\label{G(r,n)_k weak Fano}
Let $\G(r,n)_k$ be the blow-up of the Grassmannian $\G(r,n)$ at $k\geq 1$ general points. Then
\begin{itemize}
\item[-] $\G(r,n)_k$ is Fano if and only if $(r,n)=(1,3)$ and $k\leq 2$.
\item[-] $\G(r,n)_k$ is weak Fano if and only if one of the following holds:
\begin{itemize}
\item[-] $(r,n)=(1,3)$ and $k\leq 2$;
\item[-] $(r,n)=(1,4)$ and $k\leq 4$.
\end{itemize}
\end{itemize}
\end{Proposition}
\begin{proof}
We have $-K_{\G(r,n)_1}=(n+1)H-((r+1)(n-r)-1)E$ and
$$\begin{cases}
-K_{\G(r,n)_1}\cdot e= (r+1)(n-r)-1 >0; \\
-K_{\G(r,n)_1}\cdot l= (n+1)- ((r+1)(n-r)-1)=-nr+r^2+r+2.
\end{cases}$$

If $r=1$, then $-K_{\G(r,n)_1}\cdot l=-n+4$, and therefore by Lemma \ref{conecurvesgeneral}, $-K_{\G(r,n)_1}$ is nef if and only if $n\leq 4$, and ample if and only if $n=3$. Thus $\G(r,n)_1$ is not weak Fano for $n\geq 5$ and it is not Fano for $n\geq 4$.
If $r\geq 2$, we have $-K_{\G(r,n)_1}\cdot l\leq -(2r+1)r+r^2+r+2=-r^2+2<0$. Thus $\G(r,n)_1$ is not weak Fano for any $r\geq 2$.

Now, we are left with $\G(1,3)_k$ and $\G(1,4)_k$. Since $\G(1,3)\subset\mathbb{P}^5$ is a quadric the statement for $\G(1,3)_k$ follows from Lemma \ref{Q_k^n weak Fano}.
To conclude that $\G(1,4)_k$ is weak Fano if and only if $k\leq 4$ it is enough to recall that $\deg(\G(1,4))=5$, $\codim_{\mathbb{P}^9}(\G(1,4))=3$ and to apply Lemma \ref{wFdegree}.
\end{proof}

Finally, we are ready to prove Theorem \ref{main1}. 

\begin{proof}[{Proof of Theorem \ref{main1}}]
Let $X\subset \P^N$ be an $n$-dimensional prime Fano variety of index $\iota_X$. Then $-K_{X_k} = \iota_{X}H-(n-1)(E_1+\dots +E_k)$, $-K_{X_k}\cdot l_i = \iota_{X}-(n-1)$, and $X_k$ weak Fano forces $\iota_{X}\geq n-1$. Fano varieties of dimension $n$ and index $\iota_{X}\geq n-2$ have been classified. By \cite{KO73}, $\iota_{X}\leq n+1$, and equality  holds if and only if $X\cong \P^n$. Moreover, $\iota_{X}= n$ if and only if $X$ is a quadric hypersurface $Q^n\subset \P^{n+1}$. Now, our statement for $X\cong\mathbb{P}^n$ follows from \cite[Proposition 2.9]{BL12}, and for $X\cong Q^n$ from Proposition \ref{Q_k^n weak Fano}. 

Now, recall that Fano varieties with index $\iota_X= n-1$ are called \emph{del Pezzo} manifolds, and were classified in \cite{fuj82a} and  \cite{fuj82b}. The prime Fano ones are isomorphic to one of the following.
\begin{itemize}
	\item[-] A cubic hypersurface $Y_3\subset\P^{n+1}$ with $n\geq 3$.
	\item[-] An intersection of two quadric hypersurfaces in $Y_{2,2}\subset\P^{n+2}$ with $n\geq 3$.
	\item[-] A linear section of codimension $c\leq 3$ of the Grassmannian $\mathbb{G}(1,4)\subset\P^9$ under the Pl\"ucker embedding.
\end{itemize}
Now, to conclude it is enough to apply Lemma \ref{wFdegree} noticing that $\deg(Y_3)=3$, $\deg(Y_{2,2})=4$, and that any linear section of codimension $0\leq c \leq 3$ 
of $\mathbb{G}(1,4)\subset\mathbb{P}^9$ has degree $\deg(\mathbb{G}(1,4)) = 5$.
\end{proof}

\section{Spherical Grassmannian blow-ups}\label{sGbu}
In this section we study the spherical blow-ups of Grassmannians at general points listed in Theorem \ref{main2}, and compute their cones of effective divisors.

The common strategy of the proofs in this section consists in choosing a reductive group, a Borel subgroup, a point $q\in\G(r,n)$, and divisors $D_1,\dots,D_m$ in $\G(r,n)$ in such a way that the Borel orbit of $q$ in $\G(r,n)$ is the complement of $D_1\cup\dots\cup D_m.$ 

\begin{Proposition}\label{G(r,n)1 spherical}
$\G(r,n)_1$ is a spherical variety and 
$$\Eff(\G(r,n)_1)= \left\langle E,H-(r+1)E\right\rangle.$$
\end{Proposition}
\begin{proof}
In the same notation of Example \ref{G(r,n) spherical} consider $G_1:=\{g\in G;\ g\cdot p = p\}$, where $p\in \G(r,n)$ is the blown-up point. Note that $G_1$ is the set of matrices with the $(n-r)\times (r+1)$ left down block equal to zero.
This algebraic group $G_1$ is not reductive because its unipotent radical is the normal subgroup $U_1$ of matrices with the two diagonal blocks equal to the identity.
The quotient $G_1^{red}=G_1/U_1$ can be identified with the set of matrices in 
$SL_{n+1}$ with non-zero entries only in the two diagonal blocks, and the two non diagonal blocks zero.
We have an isomorphism
$$G_1^{red}\cong\{M=(M',M'')\in GL_{r+1}\times GL_{n-r};\ \det(M')\det(M'')=1\}.$$
Therefore, $G_1^{red}$ is reductive and acts on $\G(r,n)_1$.
Consider the Borel subgroup $B_1\subset G_1^{red}$ of matrices with upper triangular blocks. These groups are of the following forms:
$$
G_1\!=\!\left\{
\begin{pmatrix}
	A & B \\ 0 & C
\end{pmatrix}\right\},
U_1\!=\!\left\{
\begin{pmatrix}
	Id & B \\ 0 & Id
\end{pmatrix}\right\},
G_1^{red}\!=\!\left\{
\begin{pmatrix}
	A & 0 \\ 0 & C
\end{pmatrix}\right\},
B_1\!=\!\left\{
\begin{pmatrix}
	D & 0 \\ 0 & E
\end{pmatrix}\right\}
$$
where $A,D\!\in\! GL_{r+1},B\!\in\! M_{(r+1)\times(n-r)},C,E\!\in\! GL_{n-r},$
$D,E$ are upper triangular, and $\det(A)\cdot \det(C)=\det(D)\cdot \det(E)=1$.

Let $\pi:\G(r,n)_1\to \G(r,n)$ be the blow-up map and $E$ the exceptional divisor.
Consider the point 
$p_0 = [\left\langle e_0+e_n,\dots, e_r+e_{n-r}\right\rangle]\in \G(r,n),$
and a point $p_0'\in \G(r,n)_1$ such that $\pi(p_0')=p_0$. By \cite[Corollary II.7.15]{Ha77} the action of $G_1$ on $\G(r,n)$ induces an action of $G_1$ on $\G(r,n)_1$. 

We claim that $B_1 \cdot p_0'$ is dense in $\G(r,n)_1$. In order to prove this we consider the following linear subspaces of $\P^n$
$$\begin{cases}
\Gamma_0=\left\langle e_{r+1},\dots, e_n\right\rangle \\
\Gamma_1=\left\langle e_0,e_{r+1},\dots, e_{n-1}\right\rangle \\
\vdots\\
\Gamma_{r+1}=\left\langle e_0,\dots,e_r,e_{r+1},\dots, e_{n-r-1}\right\rangle 
\end{cases},
\quad
\begin{cases}
\Gamma_0'=\left\langle e_0,e_{r+1},\dots, e_n\right\rangle \\
\Gamma_1'=\left\langle e_0,e_1,e_{r+1},\dots, e_{n-1}\right\rangle \\
\vdots\\
\Gamma_{r}'=\left\langle e_0,\dots,e_r,e_{r+1},\dots, e_{n-r}\right\rangle 
\end{cases}.
$$
Note that
$\Gamma_j\cong \P^{n-r-1}$ and $\Gamma_j'=\left\langle\Gamma_j,e_j\right\rangle\cong \P^{n-r}$. Thus we can define divisors $D_j:=\{[\Sigma]\in \G(r,n): \Sigma \cap \Gamma_j \neq \varnothing \}$ in $\G(r,n)$, and by \cite[Lemma 7.2.1]{Ri17} the strict transform of $D_j$ has class $H-jE\in N^1(\G(r,n)_1)$ for $j=0,\dots,r+1$. Note also that the $D_j$ are $B_1$-invariant but not $G^{red}_1$-invariant prime divisor, and $p_0\notin D_0\cup D_1 \cup \dots \cup \ D_{r+1}$. Therefore, it is enough to prove that
\stepcounter{thm}
\begin{equation}\label{orbitp_0inG(r,n)1}
B_1\cdot  p_0=\G(r,n)\backslash\left\{
D_0\cup D_1 \cup \dots \cup \ D_{r+1}\right\}.
\end{equation}
Indeed, it will then follow that $B_1\cdot p_0'=\G(r,n)_1\backslash\left\{
D_0\cup D_1 \cup \dots \cup \ D_{r+1}\cup E\right\}$, and the result will follow from Proposition \ref{effcone spherical}.

Now, let $q\in \G(r,n)\backslash
\left\{
D_0\cup D_1 \cup \dots \cup \ D_{r+1}\right\}$, and consider $w_j\in \Gamma_j'\cap \Sigma_q\subset \P^n,
j=0,\dots, r$, where $\Sigma_q$ is the $r$-plane of $\P^n$ corresponding to $q$. We may write
%\stepcounter{thm}
\begin{equation*}\label{G(r,n)1 Borbit}
\begin{bmatrix}
w_0\\ w_1\\ \vdots \\w_r
\end{bmatrix}=[d\  e]=
\begin{bmatrix}
d_{0,0} &             &           &   & e_{0,r+1} &   &\dots & &e_{0,n-1} & e_{0,n}\\
d_{1,0} &  d_{1,1}     &           &   & e_{1,r+1} &  &\dots & &e_{1,n-1} & \\
\vdots   &  &    \ddots       &    &\vdots  &   & & \rddots &&\\
d_{r,0} & \ldots & \dots& d_{r,r} & e_{r,r+1} &  \dots &e_{r,n-r} & &   &  
\end{bmatrix}
\end{equation*}
for some complex numbers $d_{ij},e_{ij}$. Note that $d_{jj}=0\Rightarrow w_j\in \Gamma_j\Rightarrow q\in D_j$, and therefore $d_{jj}\neq 0$ for $j=0,\dots,r$. Hence $\Sigma_q=\left\langle w_0,w_1,\dots,w_r \right\rangle$. Similarly the $e_{ij}$'s on the diagonal are also non-zero because $e_{j, n-j}=0\Rightarrow w_j\in \Gamma_{j+1}\Rightarrow q\in D_{j+1}$. Finally, let
$$b\!=\!
\left[
\begin{tabular}{ccc}
$d^{T}$ & 0 & 0\\
0 & $Id$ & \multirow{2}{*}{$\widetilde e$}\\
0 & 0 & 
\end{tabular}
\right],
\quad
\widetilde e=
\begin{bmatrix}
e_{r,r+1} & \dots & e_{0,r+1}\\
\vdots & & \\
e_{r,n-r}&&\vdots\\
&\ddots &\\
&& e_{0,n}
\end{bmatrix}.
$$
To conclude that equation (\ref{orbitp_0inG(r,n)1}) holds it is enough to note that $q=b\cdot p_0$. 
\end{proof}

\begin{Proposition}\label{G(r,2r+1)2 spherical}
$\G(r,2r+1)_2$ is a spherical variety for any $r\geq 1$ and 
$$\Eff(\G(r,2r+1)_2)=\left\langle E_1,E_2,H-(r+1)E_1,H-(r+1)E_2\right\rangle.$$
\end{Proposition}
\begin{proof}
It is enough to proceed as in the proof of Proposition \ref{G(r,n)1 spherical} with $D_j=H-jE_1-(r+1-j)E_2$.
\end{proof}

\begin{Proposition}\label{G(1,n)2 spherical}
$\G(1,n)_2$ is a spherical variety for any $n\geq 5$ and 
$$\Eff(\G(1,n)_2)=\left\langle E_1,E_2,H-2(E_1+E_2)\right\rangle.$$
\end{Proposition}
\begin{proof}
We may assume that $p_1=[\left\langle e_0,e_1 \right\rangle], p_2=[\left\langle e_2,e_3 \right\rangle]$. Let us consider the groups
$$
G_2\!\!=\!\!\left\{
\begin{pmatrix}
	A & 0 & B\\
	0 & C & D\\
  0 & 0 & E\\	
\end{pmatrix}\right\},
G_2^{red}\!\!=\!\!\left\{
\begin{pmatrix}
	A & 0 & 0\\
	0 & C & 0\\
  0 & 0 & E\\	
\end{pmatrix}\right\},
B_2\!\!=\!\!\left\{
\begin{pmatrix}
	b_{00}&b_{01}&  0 	&  0   &  0   &  0   &  0   \\
	  0		&b_{11}&  0		&  0   &  0   &  0   &  0   \\
	  0 	&  0 	 &b_{22}&b_{23}&  0   &  0   &  0   \\
	  0		&  0	 &  0		&b_{33}&  0   &  0   &  0   \\
    0 	&  0 	 &	0 	&  0 	 &b_{44}&\dots &b_{4n}\\
    0 	&  0 	 &	0		&  0	 &  0		&\ddots&\vdots\\
		0 	&  0 	 &	0		&  0	 &  0		&  0   &b_{nn}\\
	 
\end{pmatrix}\right\},
$$
and the following linear subspaces of $\P^n$
$$\begin{cases}
\Gamma_{02}=\left\langle e_2,e_3,\dots, e_n\right\rangle \\
\Gamma_{11}=\left\langle e_0,e_2,e_4,e_5,\dots,
e_n\right\rangle \\
\Gamma_{20}=\left\langle e_0,e_1,e_4,e_5,\dots,
e_n\right\rangle \\
\Gamma_{21}=\left\langle e_0,e_1,e_2,e_4,e_5,\dots,
e_{n-1}\right\rangle \\
\Gamma_{12}=\left\langle e_0,e_2,e_3,e_4,e_5,\dots,
e_{n-1}\right\rangle \\
\Gamma_{22}=\left\langle e_0,e_1,\dots, e_{n-2}\right\rangle 
\end{cases},
\quad
\begin{cases}
\Gamma_0'=\left\langle e_0,e_2,e_3,\dots,e_n\right\rangle \\
\Gamma_1'=\left\langle e_0,e_1,e_2,e_4,e_5,\dots,e_n\right\rangle 
\end{cases}.
$$
Thus
$\Gamma_{ij}\cong \P^{n-2}$, $\Gamma_j'\cong \P^{n-1}$, and we can define divisors $D_{ij}:=\{[\Sigma]\in \G(1,n): \Sigma \cap \Gamma_{ij} \neq \varnothing \}$ with class $H-iE_1-jE_2$ which are $B_2$-invariant but not $G^{red}_2$-invariant prime divisor. Consider
$$p_0=\begin{bmatrix}
1&0&0&1&0&\dots&0&0&1\\
0&1&1&0&0&\dots&0&1&1\\
\end{bmatrix}=
\left\langle e_0+e_3+e_n,e_1+e_2+e_{n-1}+e_n\right\rangle
$$
and note that $p_0\in \G(1,n)\backslash \bigcup_{i,j} D_{ij}$. It is enough to show that $B_2\cdot  p_0=\G(1,n)\backslash \bigcup_{i,j} D_{ij}$. Let $q\in \G(1,n)\backslash \bigcup_{i,j} D_{ij}$ and choose $w_j\in \Sigma_q \cap\Gamma_j'$, where $\Sigma_q\subset \P^n$ is the line corresponding to $q\in \G(1,n)$. Then we have
$$\begin{bmatrix}
w_0 \\ w_1
\end{bmatrix}=\begin{bmatrix}
x_0&0&x_2&x_3&x_4&\dots&x_{n-2}&x_{n-1}&x_n\\
y_0&y_1&y_2&0&y_4&\dots&y_{n-2}&y_{n-1}&y_n\\
\end{bmatrix}.
$$
Now, note that 
\stepcounter{thm}
\begin{equation}\label{eqdivesf}
\begin{cases}
q\notin D_{02} \Rightarrow x_0\neq 0;\\ 
q\notin D_{11} \Rightarrow y_1\neq 0, x_3\neq 0;\\ 
q\notin D_{20} \Rightarrow y_2\neq 0;
\end{cases}
\begin{cases}
q\notin D_{21} \Rightarrow y_n\neq 0;\\ 
q\notin D_{12} \Rightarrow x_n\neq 0;\\
q\notin D_{22} \Rightarrow x_{n-1}y_n-x_ny_{n-1}\neq 0.
\end{cases}.
\end{equation}
This yields that $\Sigma_q=\left\langle w_0,w_1 \right\rangle$, and up to re-scaling either $w_0$ or $w_1$ we may assume that $x_n=y_n\neq 0$. Then using the last condition in (\ref{eqdivesf}) we have that $y_{n-1}\neq x_{n-1}$ and thus considering
$$b\!=\!
\left[
\begin{tabular}{cccc}
$b_1$&0 & 0 & 0\\
0&$b_2$ & 0 & 0\\
0&0 & $Id$ & \multirow{2}{*}{$b_3$}\\
0&0 & 0 & 
\end{tabular}
\right]\in B_2,
b_1=
\begin{bmatrix}
x_0&y_0\\ 0 &y_1	 
\end{bmatrix},
b_2=
\begin{bmatrix}
y_2&x_2\\ 0 &x_3
\end{bmatrix},
b_3=
\begin{bmatrix}
y_4-x_4&x_4\\
y_5-x_5&x_5\\
\vdots&\vdots\\
y_{n-1}-x_{n-1}&x_{n-1}\\
0&x_n\\
\end{bmatrix}
$$
we have $q=b\cdot p_0$.
\end{proof}

\begin{Proposition}\label{G(1,5)3 spherical}
$\G(1,5)_3$ is a spherical variety and 
$$\Eff(\G(1,5)_3)=\left\langle E_1,E_2,E_3,H-2(E_1+E_2),H-2(E_1+E_3),H-2(E_2+E_3)\right\rangle.$$ 
\end{Proposition}
\begin{proof}
It is enough to argue as in the proof of Proposition \ref{G(1,n)2 spherical} with $D_{ij}=H-iE_1-jE_2-(4-i-j)E_3$.
\end{proof}

\begin{Proposition}\label{G(r,2r+2)2 spherical}
$\G(r,2r+2)_2$ is spherical for any $r\geq 1$ and 
$$\Eff(\G(r,2r+2)_2)=\left\langle E_1,E_2,H-(r+1)E_1-E_2,H-E_1-(r+1)E_2\right\rangle.$$ 
\end{Proposition}
\begin{proof}
Consider the points $p_1=[\left\langle e_0,e_1,\dots,e_r \right\rangle],
p_2=[\left\langle e_{r+1},e_{r+2},\dots,e_{2r+1} \right\rangle]$, the group
$$
B_2\!=\!\left\{
\begin{pmatrix}
	A_1 & 0  & 0  \\
	0  & A_2 & 0  \\
  0  & 0   & A_3\\	
\end{pmatrix}\right\},
\mbox{ where } A_3\in \C^* ,A_i=\begin{bmatrix}
a_{0,0}^i & \ldots & a_{0,r}^i	\\
          & \ddots & \vdots   		\\
          &        & a_{r,r}^i
\end{bmatrix}
\mbox{ for } i=1,2,$$
and the following linear subspaces of $\P^{2r+2}$
\begin{small}
$$\begin{cases}
\Gamma_{0, r+1}=\left\langle e_{r+1},\dots, e_{2r+2}\right\rangle \\
\Gamma_{1,r}=\left\langle e_0,e_{r+1},\dots, e_{2r},e_{2r+2}\right\rangle \\
\vdots\\
\Gamma_{r+1, 0}=\left\langle e_0,\dots,e_r,e_{2r+2}\right\rangle 
\end{cases}\!\!\!
\begin{cases}
\Gamma_{1, r+1}=\left\langle e_0,e_{r+1},\dots,e_{2r+1}\right\rangle \\
\Gamma_{2,r}=\left\langle e_0,e_1,e_{r+1},\dots,e_{2r}\right\rangle \\
\vdots\\
\Gamma_{r+1, 1}=\left\langle e_0,\dots,e_{r+1}\right\rangle 
\end{cases}\!\!\!
\begin{cases}
\Gamma_0'=\left\langle e_0,e_{r+1},\dots, e_{2r+2}\right\rangle \\
\Gamma_1'=\left\langle e_0,e_1,e_{r+1},\dots, 
e_{2r},e_{2r+2}\right\rangle \\
\vdots\\
\Gamma_{r}'=\left\langle e_0,\dots,e_r,e_{r+1},e_{2r+2}\right\rangle 
\end{cases}.
$$
\end{small}
To conclude it is enough to take
$$p_0\!=\!\begin{pmatrix}
1  &  &  &&      && 1   & 1 \\
  & \ddots&  &&&\rddots&   & \vdots \\
 &  & 1 && 1     &    && 1 
\end{pmatrix}\in \G(r,2r+2)$$
and to argue as in the proof of Proposition \ref{G(1,n)2 spherical}.
\end{proof}

\section{Non-spherical blow-ups of Grassmannians}
\label{nsGbu}

In this section we show that the blow-ups of Grassmannians at general points in the previous section are the only spherical ones. We also describe which blow-ups have complexity at most one. Recall that by Remark \ref{rem1} these blow-ups are Mori dream spaces.

\setlength{\tabcolsep}{20pt}
\begin{table}[h!]
\centering
%\caption{Complexity one blow-ups of $\G(r,n)$ at general points}
\label{tablec1}
\begin{tabular}{ccc}
\hline
Number of blown-up points & Complexity one blow-ups & Reference\\
\hline
$1$    & none      
& Proposition \ref{G(r,n)1 spherical}                  \\ 
$2$    & $\G(2,n)_2, n\geq 7$      
& Corollary \ref{corollary k=2 c=1}                  \\ 
$3$    & $\G(1,n)_3, n\geq 6$ and $\G(2,8)_3$                                                      & Corollary \ref{corollary k>=3 c=1}  \\ 
$4$    &  $\G(1,7)_4$ 
& Corollary \ref{corollary k>=3 c=1} \\ 
$\geq 5$ & none  
& Corollary \ref{corollary k>=3 c=1}   \\ \hline 
\end{tabular}
\end{table}

Unlike the previous section, here we do not intend to exhibit the Borel orbit of a general point but rather to compute its codimension. 
The first step consists in relating the automorphisms of a variety to the automorphisms of its blow-ups.
Thanks to a result due to M. Brion \cite{Br11} in the algebraic setting, and to A. Blanchard \cite{Bl56} in the analytic setting we get the following result on the connected component of the identity of the automorphism group of a blow-up.

\begin{Proposition}\label{automorphism prop}
Let $X$ be a noetherian integral normal scheme, $Z\subset X$ a closed and reduced subscheme of codimension greater than or equal to two, and $X_{Z}:= Bl_{Z}X$ the blow-up of $X$ along $Z$. Then the connected component of the identity of the automorphism group of $X_Z$ is isomorphic to the connected component of the identity of the subgroup $\Aut(X,Z)\subset \Aut(X)$ of automorphisms of $X$ stabilizing $Z$, that is
$$\Aut(X_Z)^{o}\cong \Aut(X,Z)^{o}.$$ 
\end{Proposition}
\begin{proof}
Let $\pi:X_{Z}\to X$ be the blow-up of $X$ along $Z$. Since $\pi_{*}\mathcal{O}_{X_{Z}}\cong \mathcal{O}_X$, by \cite[Proposition 2.1]{Br11} any automorphism $\phi\in \Aut(X_Z)^{o}$ induces an automorphism $\overline{\phi}\in \Aut(X)$ such that the diagram 
\[
  \begin{tikzpicture}[xscale=2.0,yscale=-1.2]
    \node (A0_0) at (0, 0) {$X_Z$};
    \node (A0_1) at (1, 0) {$X_Z$};
    \node (A1_0) at (0, 1) {$X$};
    \node (A1_1) at (1, 1) {$X$};
    \path (A0_0) edge [->]node [auto] {$\scriptstyle{\phi}$} (A0_1);
    \path (A1_0) edge [->]node [auto] {$\scriptstyle{\overline{\phi}}$} (A1_1);
    \path (A0_1) edge [->]node [auto] {$\scriptstyle{\pi}$} (A1_1);
    \path (A0_0) edge [->]node [auto,swap] {$\scriptstyle{\pi}$} (A1_0);
  \end{tikzpicture}
\]
commutes. Let $x\in Z$ be a point such that $\overline{\phi}(x)\notin Z$, and let $F_x, F_{\overline{\phi}(x)}$ be the fibers of $\pi$ over $x$ and $\phi(x)$ respectively. Then $\phi_{|F_x}:F_x\rightarrow F_{\overline{\phi}(x)}$ induces an isomorphism between $F_x$ and $F_{\overline{\phi}(x)}$. On the other hand $F_x$ has positive dimension while $F_{\overline{\phi}(x)}$ is a point. A contradiction. Therefore $\overline{\phi}\in \Aut(X,Z)$.  Furthermore, since $\phi\in \Aut(X_Z)^{o},$ the automorphism $\overline{\phi}$ must lie in $\Aut(X,Z)^{o}$. This yields a morphism of groups 
$$
\begin{array}{cccc}
\chi: &\Aut(X_Z)^{o}& \longrightarrow & \Aut(X,Z)^{o}\\
      & \phi & \longmapsto & \overline{\phi}
\end{array}
$$
If $\overline{\phi} = Id_{X}$ then $\phi$ coincides with the identity on a dense open subset of $X_Z$, hence $\phi = Id_{X_Z}$. Therefore, the morphism $\chi$ is injective. 
Finally, by \cite[Corollary 7.15]{Ha77} any automorphism of $X$ stabilizing $Z$ lifts to an automorphism of $X_Z$, that is $\chi$ is surjective as well. 
\end{proof}

Given $k$ general points in $\G(r,n)$ corresponding to linear subspaces $p_1,\dots,p_k\subset \P^n,$ denote by $G_k$ the group 
$$G_k=\{g\in SL_{n+1}: gp_1=p_1,\dots,gp_k=p_k\}\subset SL_{n+1},$$ 
by $U_k$ its unipotent radical, by $G_k^{red}$ the quotient $G_k/U_k$, and by $B_k$ a Borel subgroup of $G_k^{red}$. 

\begin{Lemma}\label{Bkisenough}
If $\G(r,n)_k$ is a $G$-variety with complexity $c$ for some reductive group $G$, 
then $\G(r,n)_k$ is a $G_k^{red}$-variety (see notation above) with complexity $c_k\leq c$. In particular, $\G(r,n)_k$ is spherical if and only if $B_k$ has a dense orbit. Moreover, the complexity can not decrease under blow-up, that is $c_k\geq c_{k-1}$. In particular, $\G(r,n)_k$ spherical implies $\G(r,n)_{k-1}$ spherical.
\end{Lemma}
\begin{proof}
Assume that $\G(r,n)_k$ is a $G$-variety with complexity $c$. Then $G\subset \Aut^\circ(\G(r,n)_k)$ is a reductive group with a Borel subgroup $B$ having general orbit of codimension $c$. By Proposition \ref{automorphism prop} we have
$$\Aut^\circ(\G(r,n)_k)\cong\Aut^\circ(\G(r,n),p_1,\dots,p_k),$$
where $p_1,\dots,p_k\in \G(r,n)$ are general points, and $\Aut^\circ(\G(r,n),p_1,\dots,p_k)$ is the connected component of the identity of the group of automorphisms of $\G(r,n)$ fixing the $p_i$'s. By \cite[Theorem 1.1]{Ch49} $\Aut^\circ(\G(r,n))\cong \P(GL_{n+1})$. Therefore $\Aut^\circ(\G(r,n)_k)\cong
(\P(GL_{n+1}),p_1,\dots,p_k)$. Since $G$ is a reductive affine algebraic group, we may assume that $G\subset (SL_{n+1},p_1,\dots,p_k)=G_k$.

Therefore, after conjugation if necessary, we have $B\subset B_k$ and the codimension $c$ of a general $B$-orbit is at least the codimension $c_k$ of a general $B_k$-orbit. Finally, the last statement follows from Proposition \ref{automorphism prop}. 
\end{proof}

When $k\leq \left\lfloor \frac{n+1}{r+1} \right\rfloor$
we may choose the points as
$$p_1=[\left\langle e_0,\dots,e_r \right\rangle],\dots,p_k=[\left\langle e_{(k-1)(r+1)},\dots,e_{k(r+1)-1}\right\rangle]$$
and the corresponding Borel subgroup $B_k$ with upper triangular blocks. We start by analyzing what happens for the blow-up at two general points.

\subsection{Blow-up at two general points}
In Section \ref{sGbu} we showed that $\G(1,n)_2$ is spherical. In order to determine whether, for some fixed value of $r\geq 2$, $\G(r,n)_2$ is spherical we begin with the smallest possible $n = 2r+1$ and keep increasing $n$. When $n = 2r+1$ the dimension of the Borel subgroup $B_2$ is greater than the dimension of $\G(r,2r+1)$, and therefore we may have a dense orbit. Indeed, by Proposition \ref{G(r,2r+1)2 spherical} this is the case.

When we consider the next $n$, that is $n=2r+2,$ then $\dim(B_2)=\dim(\G(r,n))$ and $\G(r,n)_2$
may be spherical. Again this is the case as we saw in Proposition \ref{G(r,2r+2)2 spherical}. Now, when $r\geq 2$ we must take into account the gap $2r+2<n<4r+1$. Note that when $r=1$ this gap does not exist.

\begin{Lemma}\label{dimcount}
If $2r+2<n<4r+1$ then 
$$c(\G(r,n)_2)\geq 
\dfrac{(n-(2r+2))((4r+1)-n)}{2}>0.$$
In particular, if $2r+2<n<4r+1$ then $\G(r,n)_2$ is not spherical.
\end{Lemma}
\begin{proof}
This is just a dimension count:
$c(\G(r,n)_2)\geq \dim(\G(r,n))-\dim(B_2)$
$$=(n-r)(r+1)- \left[(r+1)(r+2)+\dfrac{(n-2r-1)(n-2r)}{2}-1\right]=
\dfrac{(n-2r-2)(4r+1-n)}{2}.$$
The last claim follows from Lemma \ref{Bkisenough}.
\end{proof}

To conclude the analysis of the case $k = 2$ it is enough to apply the following result.

\begin{Proposition}\label{G(r,n)2 spherical}
If $r\geq 2$ and $n\geq 2r+2$ then
$$c(\G(r,n)_2)=\begin{cases}
\dfrac{(n-(2r+2))((4r+1)-n)}{2} & \mbox{ if } n\leq 3r+2\\
\dfrac{r(r-1)}{2} & \mbox{ if } n> 3r+2
\end{cases}.$$
In particular, $\G(r,n)_2$ is not spherical for $r\geq 2$ and $n\geq 4r+1$.
\end{Proposition}
\begin{proof}
In this case, as in the proof of Proposition \ref{G(r,2r+2)2 spherical}, $B_2$ has three blocks but the last one is bigger:
$$A_3=\begin{bmatrix}
a_{0,0}^3 & a_{0,1}^3 & \ldots & a_{0,l}^3	\\
          & a_{1,1}^3 & \ldots & a_{1,l}^3  \\
          &           & \ddots & \vdots   		\\
          &           &a_{l-1,l-1}^3& a_{l-1,l}^3\\
  0       &           &        & a_{l,l}^3
\end{bmatrix}, \mbox{where } l=n-(2r+2).
$$
We consider linear spaces
$$\begin{cases}
\Gamma_0'=\left\langle e_0,e_{r+1},\dots, e_{n}\right\rangle \\
\Gamma_1'=\left\langle e_0,e_1,e_{r+1},\dots, 
e_{2r},e_{2r+2},\dots,e_n\right\rangle \\
\vdots\\
\Gamma_{r}'=\left\langle e_0,\dots,e_r,e_{r+1},e_{2r+2},\dots,e_n\right\rangle 
\end{cases}
$$
of codimension $r$ in $\P^n$. Now given a general $q\in \G(r,n)$ corresponding to
$\Sigma_q$ our goal is to compute the dimension of the stabilizer of $q$ under the action of $B_2$.
There is only one point $w_j$ in each intersection $\Gamma_j'\cap \Sigma_q$ and
$\Sigma_q=\left\langle w_0,\dots,w_r\right\rangle$. Write $w_i=(x^i_0:\dots: x^i_r:y^i_0:\dots: y^i_r:z^i_0:\dots: z^i_l)$ for $i=0,\dots, r$. Then
$$\Sigma_q=
\begin{bmatrix}
x_0^0 &  0      & 0  & y_0^0 & \dots  &y_r^0		 & z^0_0 & \dots & z^0_l\\
\vdots&    \ddots     &   0 & \vdots&    \rddots    &    	0	&\vdots &       &     \vdots \\
x_0^r &  \ldots &x_r^r&y_0^r&  0    & 0       & z^r_0 & \dots & z^r_l    
\end{bmatrix}.
$$
Notice that $\Gamma_0',\dots,\Gamma_r'$ are stabilized by $B_2$, therefore $b\in B_2$ stabilizes $\Sigma_q$ if and only if $b$ fixes $w_0,\dots,w_r$.
Hence, setting $a_{l,l}^3=1$, we get that $b\in B_2$ stabilizes $\Sigma_q$ if only if:
\begin{footnotesize}
$$
\begin{cases}
a_{0,0}^1x_0^0															&\!=\!\lambda_0 x_0^0\\
a_{0,0}^1x_0^1+a_{0,1}^1x_1^1								&\!=\!\lambda_1 x_0^1\\
a_{0,0}^1x_0^2+a_{0,1}^1x_1^2+a_{0,2}^1x_2^2&\!=\!\lambda_2 x_0^2\\
\vdots   																		&\\
a_{0,0}^1x_0^r+\cdots +a_{0,r}^1x_r^r   		&\!=\!\lambda_r x_0^r\\
\end{cases}\
\begin{cases}
0																						&\!=\!0							 \\
a_{1,1}^1x_1^1															&\!=\!\lambda_1 x_1^1\\
a_{1,1}^1x_1^2+a_{1,2}^1x_2^2								&\!=\!\lambda_2 x_1^2\\
\vdots   																		&\\
a_{1,1}^1x_1^r+\cdots +a_{1,r}^1x_r^r   		&\!=\!\lambda_r x_1^r\\
\end{cases}\
\dots
\begin{cases}
0																						&\!=\!0							 \\
0																						&\!=\!0							 \\
\vdots   																		&\\
0																						&\!=\!0							 \\
a_{r,r}^1x_r^r  													  &\!=\!\lambda_r x_r^r
\end{cases}
$$
$$
\begin{cases}
a_{0,0}^2y_0^0+\cdots +a_{0,r-1}^2y_{r-1}^0 +a_{0,r}^2y_r^0&\!=\!\lambda_0 y_0^0\\
a_{0,0}^2y_0^1+\cdots +a_{0,r-1}^2y_{r-1}^1	&\!=\!\lambda_1 y_0^1\\
\vdots   																		&\\
a_{0,0}^2y_0^{r-1}+a_{0,1}^2y_1^{r-1}					&\!=\!\lambda_{r-1} y_0^{r-1}\\
a_{0,0}^2y_0^r 															&\!=\!\lambda_r y_0^r
\end{cases}
\dots
\begin{cases}
a_{r-1,r-1}^2y_{r-1}^0+a_{r-1,r}^2y_r^0			&\!=\!\lambda_0 y_{r-1}^0\\
a_{r-1,r-1}^2y_{r-1}^1									  	&\!=\!\lambda_1 y_{r-1}^1\\
0																						&\!=\!0\\							
\vdots   																		&\\
0																						&\!=\!0							
\end{cases}
\begin{cases}
a_{r,r}^2y_r^0  													  &\!=\!\lambda_0 y_r^0\\
0																						&\!=\!0							 \\
\vdots   																		&\\
0																						&\!=\!0							 \\
0																						&\!=\!0							 
\end{cases}
$$
$$
\begin{cases}
a_{0,0}^3z^0_0+\dots +a_{0,l}^3z^0_l&=\lambda_0 z_0^0\\
a_{0,0}^3z^1_0+\dots +a_{0,l}^3z^1_l&=\lambda_1 z_0^1\\
\vdots &\\
a_{0,0}^3z^r_0+\dots +a_{0,l}^3z^r_l&=\lambda_r z_0^r
\end{cases}
\dots
\begin{cases}
a_{l-1,l-1}^3z^0_{l-1}+a_{l-1,l}^3z^0_l&=\lambda_0 z_{l-1}^0\\
a_{l-1,l-1}^3z^1_{l-1}+a_{l-1,l}^3z^1_l&=\lambda_1 z_{l-1}^1\\
\vdots &\\
a_{l-1,l-1}^3z^r_{l-1}+a_{l-1,l}^3z^r_l&=\lambda_r z_{l-1}^r
\end{cases}
\begin{cases} 
z^0_l&=\lambda_0 z^0_l \\
z^1_l&=\lambda_1 z^1_l \\
\vdots&\\
z^r_l&=\lambda_r z^r_l
\end{cases}
$$\end{footnotesize}
for some $\lambda_0,\dots\lambda_r\in \C^*$. Therefore, from the last system of equations we get $\lambda_0=\dots=\lambda_r=1$. Then by the first systems we get $a_{ij}^1=a_{ij}^2=\delta_{ij}$
for every $i,j=0,\dots, r$.
Now, we are left with the following $l$ systems with $r+1$ equations in the variables $a_{i,j}^3$:
\begin{footnotesize}
$$
\begin{cases}
a_{0,0}^3z^0_0+\dots +a_{0,l}^3z^0_l&=z_0^0\\
a_{0,0}^3z^1_0+\dots +a_{0,l}^3z^1_l&=z_0^1\\
\vdots &\\
a_{0,0}^3z^r_0+\dots +a_{0,l}^3z^r_l&=z_0^r
\end{cases}
\begin{cases}
a_{1,1}^3z^0_1+\dots +a_{1,l}^3z^0_l&=z_1^0\\
a_{1,1}^3z^1_1+\dots +a_{1,l}^3z^1_l&=z_1^1\\
\vdots &\\
a_{1,1}^3z^r_1+\dots +a_{1,l}^3z^r_l&=z_1^r
\end{cases}
\dots
\begin{cases}
a_{l-1,l-1}^3z^0_{l-1}+a_{l-1,l}^3z^0_l&=z_{l-1}^0\\
a_{l-1,l-1}^3z^1_{l-1}+a_{l-1,l}^3z^1_l&=z_{l-1}^1\\
\vdots &\\
a_{l-1,l-1}^3z^r_{l-1}+a_{l-1,l}^3z^r_l&=z_{l-1}^r
\end{cases}.
$$
\end{footnotesize}
If $l\leq r$, that is if $n\leq 3r+2$, then 
$a_{ij}^3=\delta_{ij}$ for $i,j\leq l$, thus the dimension of the stabilizer is zero and the result follows from the dimension count in the proof of Lemma \ref{dimcount}.\\
If $l>r$, then the last $r$ systems yield $a_{ij}^3=\delta_{ij}$ for $i,j\geq l-r$. We get then $l-r$  independent systems, all of them with more variables than equations:
\begin{footnotesize}
$$
\begin{cases}
a_{0,0}^3z^0_0+\dots +a_{0,l}^3z^0_l&\!=\!z_0^0\\
a_{0,0}^3z^1_0+\dots +a_{0,l}^3z^1_l&\!=\!z_0^1\\
\vdots &\\
a_{0,0}^3z^r_0+\dots +a_{0,l}^3z^r_l&\!=\!z_0^r
\end{cases}\dots
\begin{cases}
a_{l-r-1,l-r-1}^3z^0_{l-r-1}+\dots +a_{l-r-1,l}^3z^0_l&\!=\!z_{l-r-1}^0\\
a_{l-r-1,l-r-1}^3z^1_{l-r-1}+\dots +a_{l-r-1,l}^3z^1_l&\!=\!z_{l-r-1}^1\\
\vdots &\\
a_{l-r-1,l-r-1}^3z^r_{l-r-1}+\dots +a_{l-r-1,l}^3z^r_l&\!=\!z_{l-r-1}^r
\end{cases}.
$$
\end{footnotesize}
Since the $z_i^j$ are general each system has linearly independent equations. The first system has $l+1$ variables and $r+1$ conditions, the second has $l$ variables and $r+1$ conditions, and so on up to the last system which has $r+2$ variables and $r+1$ conditions. Therefore, the dimension of the stabilizer is given by 
$$\mbox{\# of variables} - \mbox{\# of conditions}\!=\!
(l-r)\!+\!(l-r-1)\!+\!\dots+1\!=
\!\dfrac{(l-r)(l-r+1)}{2}.$$
Then the dimension of the orbit is
$$(r+1)(r+2)+\dfrac{(l+1)(l+2)}{2}-1-\dfrac{(l-r)(l-r+1)}{2}=\dim(\G(r,n))-\dfrac{r}{2}(r-1).$$
For the last claim note that we have $4r+1> 3r+2$ whenever $r\geq 2$.
\end{proof}

Using the results in Section \ref{sGbu}, Lemma \ref{dimcount} and Proposition  \ref{G(r,n)2 spherical} we get the following.

\begin{Corollary}\label{corollary k=2 c=1}
For the complexity of
$\G(r,n)_2$ we have
$$c(\G(r,n)_2)\begin{cases}
=0 & \mbox{if } \mbox{either } r=1 \mbox{ or }  n=2r+1 \mbox{ or } n=2r+2\\
=1 & \mbox{if } r=2 \mbox{ and } n\geq 7  \\
>1 & \mbox{otherwise } 
\end{cases}.$$
\end{Corollary}

\subsection{Blow-up at three or more general points}
Now we take into account the blow-ups at three or more general points. By Corollary \ref{corollary k=2 c=1} we know that $c(\G(r,n)_2)>1$ for $r\geq 3$ whenever $n>2r+2$. Since by Lemma \ref{Bkisenough} the complexity can not decrease under blow-up, Proposition~\ref{G(r,2r+2)3 spherical} below ensures that $c(\G(r,n)_3)\leq 1$ forces $r\in \{1,2\}$. Similarly, we will estimate $c(\G(r,n)_4)$ and $c(\G(r,n)_5)$ only when $c(\G(r,n)_3)\leq 1$.

\begin{Proposition}\label{G(r,2r+2)3 spherical}
For any $r$ we have $c(\G(r,2r+1)_3)\geq  r+1$, $c(\G(r,2r+2)_3)\geq 2r+1$. Furthermore, if $r\geq 2$ then $c(\G(r,2r+3)_3)\geq 3r$.
\end{Proposition}
\begin{proof}
This is a dimension count, we will compare the dimensions of $\G(r,n)$ and $B_3$. We will develop in full detail the case $n=2r+3$, the other two cases are similar. We may assume that the three general points in $\G(r,n)$ are
$$p_1=[\left\langle e_0,\dots,e_r \right\rangle],
p_2=[\left\langle e_{r+1},\dots,e_{2r+1}\right\rangle],\mbox{ and }
p_3=[\left\langle e_{2r+2}, e_{2r+3}, v_1,\dots , v_{r-1} \right\rangle],$$
where $v_1,\dots,v_{r-1}\in \C^{n+1}$ are general, and we may assume that $v_1:=e_0+\dots+e_{n}$. Consider the linear space 
$p_3'=[\left\langle e_{2r+2}, e_{2r+3}, v_1\right\rangle],$
the group
$$G_3'=\{g\in SL_{n+1}: gp_1=p_1, gp_2=p_2,gp_3'=p_3'\}\subset SL_{n+1},$$
and a Borel subgroup $B_3'$ of $G_3'^{red}=G_3'/U_3',$ where $U_3'$ is the unipotent radical of $G_3'$.
Since $G_3\subset G_3'$, up to conjugation we have $B_3\subset B_3'$. Note that $B_3'$ is the subgroup of $SL_{n+1}$ which stabilizes the three flags:
\begin{align*}
&\mathcal{F}_1: \left\langle e_0 \right\rangle \subset 
\left\langle e_0,e_1 \right\rangle \subset \dots \subset
\left\langle e_0,\dots,e_r \right\rangle\\
&\mathcal{F}_2: \left\langle e_{r+1} \right\rangle \subset 
\left\langle e_{r+1},e_{r+2} \right\rangle \subset \dots \subset
\left\langle e_{r+1},\dots,e_{2r+1} \right\rangle\\
&\mathcal{F}_3': \left\langle e_{2r+2} \right\rangle \subset 
\left\langle e_{2r+2}, e_{2r+3} \right\rangle \subset
\left\langle e_{2r+2}, e_{2r+3}, v_1 \right\rangle.
\end{align*}
Then the elements of $B_3'$ are matrices with three upper triangular blocks, and with an additional linear condition on the coefficients of each row from the second up to the $(2r+2)$-th one. Therefore
$$\dim(B_3')=\left((r+1)(r+2)+2\right)-
(2r+1)=r(r+1)+3,$$
and since $B_3 \subset B_3'$ we have $c(\G(r,2r+3)_3)\geq \dim(\G(r,2r+3))-\dim(B_3')=(r+1)(r+3)-(r(r+1)+3)=3r$.
\end{proof}

The next result completes the case $k=3,r=2$.

\begin{Proposition}\label{G(r,n)3 n small}
For the complexity of $\G(2,n)_3$ we have
$$c(\G(2,n)_3) \begin{cases}
\geq 3 & \mbox{if } n=5\\
\geq 5 & \mbox{if } n=6\\
\geq 6 & \mbox{if } n=7\\
=1 & \mbox{if } n=8\\
>1 & \mbox{if } n>8
\end{cases}.$$ 
Moreover, $c(\G(2,8)_4)>1$.
\end{Proposition}
\begin{proof}
The first three cases follow from Proposition \ref{G(r,2r+2)3 spherical}. For the fourth one it is enough to observe that since $B_3=B_2$ we have $c(\G(2,8)_3)=c(\G(2,8)_2)=1$. 

In order to compute $c(\G(2,n)_3),$ $n>8,$ we proceed as in the proof of Proposition \ref{G(r,n)2 spherical}. Here 
\begin{footnotesize}
$$
B_3\!=\!\left\{
\begin{pmatrix}
	A_1& 0   & 0   & 0   \\
	0  & A_2 & 0   & 0 \\
    0  & 0   & A_3 & 0\\	
    0  & 0   & 0   & A_4\\	
\end{pmatrix}\right\},
\mbox{where }A_i\!=\!\begin{bmatrix}
a_{0,0}^i & a_{0,1}^i & a_{0,2}^i	\\
          & a_{1,1}^i & a_{1,2}^i 	\\
          &           & a_{2,2}^i
\end{bmatrix}
\mbox{for } i\!=\!1,2,3,
\mbox{ and }
A_4\!=\!\begin{bmatrix}
a_{0,0}^4 &  \ldots & a_{0,l}^4	\\
          &    \ddots&	\vdots\\
       &             & a_{l,l}^4
\end{bmatrix}
$$
\end{footnotesize}
where $l=n-9.$
We consider three linear spaces
$$\begin{cases}
\Gamma_0'=\left\langle e_0,e_{3},\dots, e_{n}\right\rangle \\
\Gamma_1'=\left\langle e_0,e_1,e_3,e_4,e_6,\dots,e_n\right\rangle \\
\Gamma_{2}'=\left\langle e_0,e_1,e_2,e_3,e_6,\dots,e_n \right\rangle 
\end{cases}
$$
of codimension $2$ in $\P^n$, and a general $q\in \G(2,n)$ corresponding to
$$\Sigma_q=\left\langle w_0,w_1,w_2\right\rangle=
\begin{bmatrix}
x_0^0 &  0      & 0  & y_0^0 & y_1^0  &y_2^0		 & t_0^0 & t_1^0 & t_2^0
& z_0^0 & \dots & z_l^0\\
x_0^1&    x_1^1    &   0 & y_0^1&    y_1^1    &    	0	&t_0^1 &   t_1^1    &    t_2^1
& z_0^1 & \dots & z_l^1\\
x_0^2 &  x_1^2 &x_2^2&y_0^2&  0    & 0       & t^2_0 & t_1^2 & t^2_2   
&z_0^2 & \dots & z_l^2\\
\end{bmatrix}.
$$
Since $\Gamma_0',\Gamma_1',\Gamma_2'$ are stabilized by $B_2$, they are stabilized by $B_3$ as well, and therefore $b\in B_3$ stabilizes $\Sigma_q$ if and only if $b$ fixes $w_0,w_1,$ and $w_2$. Hence, setting $a_{l,l}^4=1$, we get that $b\in B_3$ stabilizes $\Sigma_q$ if only if
\begin{footnotesize}
$$
\begin{cases}
a_{0,0}^1x_0^0&\!=\!\lambda_0 x_0^0\\
a_{0,0}^1x_0^1+a_{0,1}^1x_1^1								&\!=\!\lambda_1 x_0^1\\
a_{0,0}^1x_0^2+a_{0,1}^1x_1^2+a_{0,2}^1x_2^2&\!=\!\lambda_2 x_0^2
\end{cases}\
\begin{cases}
0&\!=\!0 \\
a_{1,1}^1x_1^1	&\!=\!\lambda_1 x_1^1\\
a_{1,1}^1x_1^2+a_{1,2}^1x_2^2								&\!=\!\lambda_2 x_1^2
\end{cases}
\begin{cases}
0		&\!=\!0	 \\
0		&\!=\!0	 \\
a_{2,2}^1x_2^2   &\!=\!\lambda_2 x_2^2
\end{cases}
$$
$$
\begin{cases}
a_{0,0}^2y_0^0+a_{0,1}^2y_1^0+a_{0,2}^2y_2^0&\!=\!\lambda_0 y_0^0\\
a_{0,0}^2y_0^1+a_{0,1}^2y_1^1&\!=\!\lambda_1 y_0^1\\
a_{0,0}^1y_0^2&\!=\!\lambda_2 y_0^2
\end{cases}
\begin{cases}
a_{1,1}^2y_1^0+a_{1,2}^2y_2^0&\!=\!\lambda_0 y_1^0\\
a_{1,1}^2y_1^1&\!=\!\lambda_1 y_1^1\\
0&\!=\!0
\end{cases}
\begin{cases}
a_{2,2}^2 y_2^0 &\!=\!\lambda_0 y_2^0\\
0		&\!=\!0	 \\
0		&\!=\!0	 
\end{cases}
$$
$$
\begin{cases}
a_{0,0}^3t^0_0+a_{0,1}^3t^0_1 +a_{0,2}^3t^0_2&\!=\!\lambda_0 t_0^0\\
a_{0,0}^3t^1_0+a_{0,1}^3t^1_1 +a_{0,2}^3t^1_2&\!=\!\lambda_1 t_0^1\\
a_{0,0}^3t^2_0+a_{0,1}^3t^2_1+a_{0,2}^3t^2_2&\!=\!\lambda_2 t_0^2
\end{cases}
\begin{cases}
a_{1,1}^3t^0_1 +a_{1,2}^3t^0_2&\!=\!\lambda_0 t_1^0\\
a_{1,1}^3t^1_1 +a_{1,2}^3t^1_2&\!=\!\lambda_1 t_1^1\\
a_{1,1}^3t^2_1+a_{1,2}^3t^2_2&\!=\!\lambda_2 t_1^2
\end{cases}
\begin{cases}
a_{2,2}^3t^0_2&\!=\!\lambda_0 t_2^0\\
a_{2,2}^3t^1_2&\!=\!\lambda_1 t_2^1\\
a_{2,2}^3t^2_2&\!=\!\lambda_2 t_2^2
\end{cases}
$$
$$
\begin{cases}
a_{0,0}^4z^0_0+\dots +a_{0,l}^4z^0_l&\!=\!\lambda_0 z_0^0\\
a_{0,0}^4z^1_0+\dots +a_{0,l}^4z^1_l&\!=\!\lambda_1 z_0^1\\
a_{0,0}^4z^2_0+\dots +a_{0,l}^4z^2_l&\!=\!\lambda_2 z_0^2
\end{cases}
\dots
\begin{cases}
a_{l-1,l-1}^4z^0_{l-1}+a_{l-1,l}^4z^0_l&\!=\!\lambda_0 z_{l-1}^0\\
a_{l-1,l-1}^4z^1_{l-1}+a_{l-1,l}^4z^1_l&\!=\!\lambda_1 z_{l-1}^1\\
a_{l-1,l-1}^4z^2_{l-1}+a_{l-1,l}^4z^2_l&\!=\!\lambda_2 z_{l-1}^2
\end{cases}
\begin{cases} 
z^0_l&\!=\!\lambda_0 z^0_l \\
z^1_l&\!=\!\lambda_1 z^1_l \\
z^2_l&\!=\!\lambda_2 z^2_l
\end{cases}
$$\end{footnotesize}
for some $\lambda_0,\lambda_1,\lambda_2\in \C^*$. Therefore, from the last system of equations we get $\lambda_0=\dots=\lambda_r=1$. Then by the first nine systems we get $a_{ij}^1=a_{ij}^2=a_{ij}^3=\delta_{ij}$ for every $i,j=0,1,2$. Now, we are left with the following $l$ systems with $3$ equations in the variables $a_{i,j}^4$
\begin{footnotesize}
$$
\begin{cases}
a_{0,0}^4z^0_0+\dots +a_{0,l}^4z^0_l&=z_0^0\\
a_{0,0}^4z^1_0+\dots +a_{0,l}^4z^1_l&=z_0^1\\
a_{0,0}^4z^2_0+\dots +a_{0,l}^4z^2_l&=z_0^2
\end{cases}
\dots
\begin{cases}
a_{l-1,l-1}^4z^0_{l-1}+a_{l-1,l}^4z^0_l&= z_{l-1}^0\\
a_{l-1,l-1}^4z^1_{l-1}+a_{l-1,l}^4z^1_l&= z_{l-1}^1\\
a_{l-1,l-1}^4z^2_{l-1}+a_{l-1,l}^4z^2_l&=z_{l-1}^2
\end{cases}
$$
\end{footnotesize}
If $l\leq 2$, that is $n=9,10$ or $11,$ then $a_{ij}^4=\delta_{ij}$ for $i,j\leq l$. Thus the dimension of the stabilizer is zero and then 
$$ c(\G(2,n)_3)= \dim(\G(2,n))-\dim(B_3)=
3(n-2)-\left(3\cdot6+\dfrac{(l+1)(l+2)}{2}
-1\right)=
\begin{cases}
3 & \mbox{if } n=9 \\
4 & \mbox{if } n=10\\
4 & \mbox{if } n=11
\end{cases}.
$$
If $l>2$, then the last two systems yield $a_{ij}^4=\delta_{ij}$ for $i,j\geq l-2$. We get then $l-2$ independent systems
\begin{footnotesize}
$$
\begin{cases}
a_{0,0}^4z^0_0+\dots +a_{0,l}^4z^0_l&\!=\!z_0^0\\
a_{0,0}^4z^1_0+\dots +a_{0,l}^4z^1_l&\!=\!z_0^1\\
a_{0,0}^4z^r_0+\dots +a_{0,l}^4z^r_l&\!=\!z_0^2
\end{cases}\dots
\begin{cases}
a_{l-3,l-3}^4z^0_{l-3}+\dots +a_{l-3,l}^4z^0_l&\!=\!z_{l-3}^0\\
a_{l-3,l-3}^4z^1_{l-3}+\dots +a_{l-3,l}^4z^1_l&\!=\!z_{l-3}^1\\
a_{l-3,l-3}^4z^r_{l-3}+\dots +a_{l-3,l}^4z^r_l&\!=\!z_{l-3}^2
\end{cases}
$$
\end{footnotesize}
all of them with more variables than equations. Since the $z_i^j$'s are general the equations in each system are linearly independent. The first system has $l+1$ variables, the second has $l$ variables, and so on up to the last one which has four variables. Therefore, the dimension of the stabilizer is given by 
$$\mbox{\# of variables} - \mbox{\# of conditions}\!=\!
(l-2)\!+\!(l-3)\!+\!\dots+1\!=
\!\dfrac{(l-2)(l-1)}{2}.$$
Then for $n\geq 12$ the complexity is given by:
\begin{eqnarray*}
c(\G(2,n)_3)=\dim(\G(2,n))- \dim(B_3)+\dim(\mbox{stabilizer})\\
=3(n-2)-
\left(3\cdot 6+\dfrac{(l+1)(l+2)}{2}-1\right)+
\dfrac{(l-2)(l-1)}{2}=4.
\end{eqnarray*}
Finally, for $\G(2,8)_4$ proceeding as in the proof of Proposition \ref{G(r,2r+2)3 spherical} we get
\begin{footnotesize}
$$
B_4'\!=\!\left\{
\begin{pmatrix}
	A_1 & 0  & 0  \\
	0  & A_2 & 0  \\
  0  & 0   & A_3\\	
\end{pmatrix}\right\},
\mbox{where }
A_1\!=\!\begin{bmatrix}
a_{0,0}^1 & a_{0,1}^1 & a_{0,2}^1	\\
          & a_{1,1}^1 & 
a_{0,0}^1+a_{0,1}^1+a_{0,2}^1-a_{1,1}^1 	\\
          &           & a_{0,0}^1+a_{0,1}^1+a_{0,2}^1
\end{bmatrix},
$$
$$
A_2\!=\!\begin{bmatrix}
a_{0,0}^2 & a_{0,1}^2 & a_{0,0}^1+a_{0,1}^1+a_{0,2}^1-a_{0,0}^2- a_{0,1}^2	\\
          & a_{1,1}^2 & a_{0,0}^1+a_{0,1}^1+a_{0,2}^1-a_{1,1}^2  	\\
          &           & a_{0,0}^1+a_{0,1}^1+a_{0,2}^1 
\end{bmatrix}, \mbox{and }
A_3\!=\!\begin{bmatrix}
a_{0,0}^3 & a_{0,1}^3 & a_{0,0}^1+a_{0,1}^1+a_{0,2}^1 -a_{0,0}^3-a_{0,1}^3\\
          & a_{1,1}^3 & a_{0,0}^1+a_{0,1}^1+a_{0,2}^1- a_{1,1}^3 	\\
          &           & a_{0,0}^1+a_{0,1}^1+a_{0,2}^1
\end{bmatrix}.
$$

\end{footnotesize}
Therefore $\dim(B_4')=9$ and $c(\G(2,8)_4)\geq 18-9=9$.
\end{proof}

Now, we are left only with the case $\G(1,n)_k, k\geq 3$.

\begin{Proposition}\label{G(1,n)3 spherical}
For the complexity of $\G(1,n)_3$ and $\G(1,n)_4$ we have
$$c(\G(1,n)_3)\begin{cases}
\geq 2 & \mbox{if } n=3\\
\geq 3 & \mbox{if } n=4\\
=0 & \mbox{if } n=5\\
=1 & \mbox{if } n\geq 6
\end{cases},\quad\quad
c(\G(1,n)_4)\begin{cases}
\geq 5 & \mbox{if } n=5\\
\geq 6 & \mbox{if } n=6\\
=1 & \mbox{if } n=7\\
=2 & \mbox{if } n\geq 8
\end{cases}.$$ 
Moreover, $c(\G(1,7)_5)\geq 8$.
\end{Proposition}
\begin{proof}
The cases $\G(1,3)_3$ and $\G(1,4)_3$ follow from Proposition \ref{G(r,2r+2)3 spherical}. The case $\G(1,5)_3$ is in Proposition \ref{G(1,5)3 spherical}. Arguing as in the proof of Proposition \ref{G(r,n)2 spherical} we get that $c(\G(1,n)_3)=1$ for $n\geq 6$.

In the cases $\G(1,5)_4$ and $\G(1,6)_4$, proceeding as in the proof of Proposition \ref{G(r,2r+2)3 spherical}, we consider the Borel subgroups 
\begin{footnotesize}
$$B_4'=\left\{
\begin{pmatrix}
 a  &  b &  0 	&  0   &  0   &  0      \\
 0		&a+b&  0		&  0   &  0   &  0\\
 0 	&  0 	 &c&a+b-c&  0   &  0   \\
 0	&  0	 &  0		&a+b&  0      &  0\\
 0 	&  0 	 &	0   	&  0 	 &d&a+b-d \\
 0 	&  0 	 &	0		&  0	 &  0	&a+b
\end{pmatrix}\right\} , 
B_4''=\left\{
\begin{pmatrix}
 a  &  b &  0 	&  0   &  0   &  0  &0    \\
 0		&a+b&  0		&  0   &  0   &  0&0\\
 0 	&  0 	 &c&a+b-c&  0   &  0 &0  \\
 0	&  0	 &  0		&a+b&  0      &  0&0\\
 0 	&  0 	 &	0   	&  0 	 &d&a+b-d &0\\
 0 	&  0 	 &	0		&  0	 &  0	&a+b&0\\
 0 	&  0 	 &	0		&  0	 &  0	&0&e\\
\end{pmatrix}\right\}.
$$
\end{footnotesize}
Then to get the result in the cases $(r,n) = (1,5)$ and $(r,n) = (1,6)$ it is enough to observe that $\dim(B_4')=3$ for $\G(1,5)$ and $\dim(B_4'')=4$ for $\G(1,6)$. Next, note that if $(r,n)=(1,7)$ then $B_3=B_4$. Now, arguing as in the proof of Proposition \ref{G(r,n)2 spherical} we can show that $c(\G(1,n)_4)=2$ if $n\geq 8$. Finally, proceeding again as in the proof of Proposition \ref{G(r,2r+2)3 spherical}, we obtain $\dim(B_5')=4$ and $c(\G(1,7)_5)\geq 12-4=8$.
\end{proof}

We summarize the results of this subsection in the corollary below.

\begin{Corollary}\label{corollary k>=3 c=1}
For the complexity of
$\G(r,n)_k, k\geq 3$ we have
$$c(\G(r,n)_3)\begin{cases}
=0 & \mbox{if } (r,n)=(1,5)\\
=1 & \mbox{if } (r,n)=(2,8) \mbox{ or } r=1 \mbox{ and } n\geq 6  \\
>1 & \mbox{otherwise } 
\end{cases},$$
$$
c(\G(r,n)_4)\begin{cases}
=1 & \mbox{if } (r,n)=(1,7)  \\
>1 & \mbox{otherwise } 
\end{cases},$$
and  $c(\G(r,n)_k)>1$ for $k\geq 5$.
\end{Corollary}

Now we are ready to prove Theorem \ref{main3}.

\begin{proof}[{Proof of Theorem \ref{main3}}]
The statement for projective spaces $r = 0$ follows from \cite[Theorem 1.3]{CT06}. By Theorem \ref{main1} we have that $\G(1,4)_4$ is weak Fano, and hence \cite[Corollary 1.3.2]{BCHM10} yields that it is a Mori dream space.

For the other cases, observe that by Proposition \ref{G(r,n)1 spherical},  Corollary \ref{corollary k=2 c=1}, and Corollary \ref{corollary k>=3 c=1} all the $\G(r,n)_k$ listed in the statement have complexity at most one, and therefore by Remark \ref{rem1} they are Mori dream spaces.   
\end{proof}

\section{Osculating spaces and stable base locus decomposition of \texorpdfstring{$\Eff(\G(r,n)_1)$}{TEXT}}\label{mcdsbl}
Let $X$ be a smooth $n$-dimensional projective variety, and let $D$ be an effective $\mathbb{Q}$-divisor on $X$. The stable base locus $\textbf{B}(D)$ of $D$ is the set-theoretic intersection of the base loci of the complete linear systems $|sD|$ for all positive integers $s$ such that $sD$ is integral
$$\textbf{B}(D) = \bigcap_{s > 0}B(sD).$$
Since stable base loci do not behave well with respect to numerical equivalence, see for instance \cite[Example 10.3.3]{La04II}, we will assume that $h^{1}(X,\mathcal{O}_X)=0$ so that linear and numerical equivalence of $\mathbb{Q}$-divisors coincide.   

Then numerically equivalent $\mathbb{Q}$-divisors on $X$ have the same stable base locus, and the pseudo-effective cone $\overline{\Eff}(X)$ of $X$ can be decomposed into chambers depending on the stable
base locus of the corresponding linear series called \textit{stable base locus decomposition}, see \cite[Section 4.1.3]{CdFG17} for further details. 

Anyway, we will deal only with the stable base locus decomposition of $X = \G(r,n)_1$, the blow-up of $\G(r,n)$ at one point, for which $h^{1}(\G(r,n)_1,\mathcal{O}_{\G(r,n)_1})=0$ indeed holds. 

If $X$ is a Mori dream space, satisfying then the condition $h^1(X,\mathcal{O}_X)=0$, determining the stable base locus decomposition of $\Eff(X)$ is a first step in order to compute its Mori chamber decomposition. 

\begin{Remark}\label{SBLMC}
The Mori chamber decomposition is a refinement of the stable base locus decomposition. Indeed, since given an effective divisor $D$ the indeterminacy locus of the map $\phi_{D}:X\dasharrow X_D$ induced by $D$ is the stable base locus of $D$, if $\textbf{B}(D_1)\neq \textbf{B}(D_2)$ then the varieties $X_{D_1}$ and $X_{D_2}$ are not isomorphic. 
\end{Remark}

In this section we determine the stable base locus decomposition of $\Eff(\G(r,n)_1)$, and we show and how it is determined by the osculating spaces of $\G(r,n)$ at the blown-up point $p\in \G(r,n)$. We would like to mention that the connection between the behavior of the tangent and osculating spaces of $\G(r,n)$, and more generally of rational homogeneous varieties, and the birational geometry of these varieties has been pointed out in \cite{MM13}, \cite{MR16}, \cite{AMR16}, \cite{Ma16}. 

\begin{Definition}\label{oscdef}
Let $X\subset \P^N$ be a projective variety of dimension $n$, and $p\in X$ a smooth point.
Choose a local parametrization of $X$ at $p$:
$$
\begin{array}{cccc}
\phi: &\mathcal{U}\subset\mathbb{C}^n& \longrightarrow & \mathbb{C}^{N}\\
      & (t_1,\dots,t_n) & \longmapsto & \phi(t_1,\dots,t_n) \\
      & 0 & \longmapsto & p 
\end{array}
$$
For a multi-index $I = (i_1,\dots,i_n)$, set $\phi_I = \frac{\partial^{|I|}\phi}{\partial t_1^{i_1}\dots\partial t_n^{i_n}}$. For any $m\geq 0$, let $O^m_pX$ be the affine subspace of $\mathbb{C}^{N}$ centered at $p$ and spanned by the vectors $\phi_I(0)$ with  $|I|\leq m$. The $m$-\textit{osculating space} $T_p^m X$ of $X$ at $p$ is the projective closure  of  $O^m_pX$ in $\mathbb{P}^N$. Note that $T_p^0 X=\{p\}$, and $T_p^1 X$ is the usual tangent space of $X$ at $p$. When no confusion arises we will write $T_p^m$ instead of $T_p^m X$ for the $m$-osculating space of $X$ at $p$. 

The \textit{$m$-osculating dimension} of  $X$ at $p$ is $\dim(T_p^m X) = \binom{n+m}{n}-1-\delta_{m,p}$ where $\delta_{m,p}$ is the number of independent differential equations of order at most $m$ satisfied by $X$ at $p$. 
\end{Definition}

\begin{Remark}[Osculating projections]\label{oscproj}
The linear projection 
$$\Pi_{T^m_p}:X\subset\mathbb{P}^N\dasharrow X_m\subset \mathbb{P}^{N-\dim(T^m_p)-1}$$
with center $T^m_pX$ is induced by the linear system of hyperplanes of $\mathbb{P}^N$ containing $T^m_pX$.  The corresponding linear system on the blow-up $\G(r,n)_1$ of $\G(r,n)$ at $p$ is the linear system of the divisor $H-(m+1)E$. By \cite[Proposition 3.2]{MR16} we have that the rational map $\Pi_{T_{p}^m}$ is birational for every $0\leq m\leq r-1$, while $\Pi_{T_{p}^r}:\G(r,n)\dasharrow \G(r,n-r-1)$ is a fibration with fibers isomorphic to $\G(r,2r+1)$.
\end{Remark}

The next step, in order to describe the stable base locus decomposition of $\Eff(\G(r,n)_1)$, consists in understanding the stable base locus of the divisors $H-(m+1)E$.

\begin{Notation}\label{notation1}
Let $p\in \G(r,n)$ be a point and for any non negative integer $m\leq r+1$ consider the $m$-osculating space $T^m_p \G(r,n)\subset \P^N$ of $\G(r,n)$ at $p$. We denote by $R_m=R_m(p)$ the subvariety of the Grassmannian defined by $R_m:=\G(r,n)\cap T^m_p\G(r,n)$. In particular $R_0=\{p\}$.
\end{Notation}

The locus $R_m$ can be characterized in two more ways.

\begin{Lemma}\label{lemmaBS}
Choose a complete flag $\{0\}=V_0\subset V_1\subset \cdots \subset V_{n+1}=\C^{n+1}$ of linear spaces in $\C^{n+1}$,
with $V_{r+1}$ corresponding to the point $p\in \G(r,n).$ Define the following Schubert varieties in $\G(r,n)$ 
$$R_m'(p)=R_m':=\{[U]\in \G(r,n); \dim (U \cap V_{r+1})\geq r+1-m \}$$
for $m=0,1,\dots, r+1$. Moreover, for any $1\leq m\leq r+1$ define
$$R_m''(p)=R_m'':=\!\!\!\!\!\!
\bigcup_{\substack{C\mbox{ \tiny{rational normal curve} }\\ p\in C\subset \G(r,n) \\ \deg(C)\leq m}}\!\!\! \!\!\! C$$
as the locus swept out by degree $m$ rational curves contained in $\G(r,n)$ and passing through $p$. Then $R_m=R_m'=R_m''$ for any $1\leq m\leq r+1$.
\end{Lemma}
\begin{proof}
First we prove that $R_{r+1}''=\G(r,n).$ Let $p,q$ be two points in $\G(r,n)$ corresponding to linear $r$-spaces $V_p$ and $V_q$ in $\P^n$ such that $\dim(V_p\cap V_q) = d$ with $-1\leq d\leq r-1$. Let $L = \left\langle V_p,V_q\right\rangle\cong\mathbb{P}^{2r-d}$. Then $p,q\in \mathbb{G}(r,L)\subset\mathbb{G}(r,n)$. Now, consider $p^{*},q^{*}\in \mathbb{G}(r-d-1,L^{\vee})\cong \mathbb{G}(r-d-1,2r-d)$ corresponding to the linear $(r-d-1)$-spaces $V_p^{*},V_q^{*}\subset L^{\vee}$. Take lines $L_0,\dots,L_{r-d-1}$ in $L^{\vee}$ that intersect $V_p^{*}$ and $V_q^{*}$, and such that they generate a $(2r-2d-1)$-plane. Let $a=L_0\cap V_p^{*}$ and $b=L_0\cap V_q^{*}$. Choose isomorphisms $\phi_i:L_0\to L_i$ for $i=1,\dots,r-d-1$ such that $\phi_i(a)=L_i \cap V_p^{*}$ and $\phi_i(b)=L_i \cap V_q^{*}$, and consider the degree $r-d$ rational normal scroll 
$$S=\bigcup_{x\in L_0}\left\langle x,
\phi_1(x),\dots,\phi_{r-d-1}(x)\right\rangle.$$
Then $C^{*}=\{[V]\in \G(r-d-1,L^{\vee});V\subset S\}$ is a rational normal curve in $\G(r-d-1,L^{\vee})$ of degree $r-d=\deg(S)$ passing through $p^{*},q^{*}\in \G(r-d-1,L^{\vee})$. Hence, dually we get a rational normal curve $C\subset \G(r,L)\subset\G(r,n)$ of degree $\deg(C) = r-d\leq r+1$ passing through $p,q\in \G(r,n)$.

Next we prove that $R_m'=R_m''$ for $1\leq m\leq r+1$. Let $q\in R_m'=R_m'(p)$. Then by the definition of $R_m'$ we have $\dim(V_p\cap V_q)=r-d$ with $d\leq m$. Therefore, there is a subspace $L\cong \P^{r+d}$ containing both $V_p$ and $V_q$. Thus, there exists a Grassmannian $\G(r,r+d)\cong \G(d-1,r+d)$ in $\G(r,n)$ such that
$p,q\in \G(d-1,r+d)$. Then, by the first part of the proof there is a rational normal curve in $\G(d-1,r+d)\subset \G(r,n)$ of degree $d\leq m$ passing through $p$ and $q$, that is $q\in R_m''$. 

Conversely, if $q\in R_m''$ then there is a rational normal curve $C$ of degree $d\leq m$
passing through $p$ and $q$. Set $X=\bigcup_{[V]\in C} V\subset \P^n$. Then $X$ is a rational normal scroll of degree $d$ and dimension $r+1$. Let $L=\P^s$ be the span of $X$ in $\P^n$. Therefore $X$ is a non degenerate subvariety of $L$ of minimal degree, that is $\deg(X)=\dim(L)-\dim(X)+1$. For details 
about rational normal scrolls and varieties of minimal degree see \cite[Chapter 19]{Ha95}.
This means that $V_p$ and $V_q$ are contained in $L\cong \P^{r+d}$. Hence $\dim(V_p\cap V_q)\geq r-d\geq r-m$ and $q\in R_m'$. 

Now, it is enough to show that $R_m=R_m'$. This follows from the description of $T^m_p\G(r,n)$ in \cite[Proposition 2.3]{MR16}.
\end{proof}

\begin{Corollary}\label{dimBS}
We have that $R_m$ is irreducible of dimension $dim(R_m)=m(n+1-m)$ for $m=0,\dots, r+1$. In particular, $R_m$ is a divisor of $\G(r,n)$ if and only if $m=r$ and $n=2r+1$.
\end{Corollary}
\begin{proof}
The statement follows from the description of $R_m$ as a Schubert variety in Lemma \ref{lemmaBS} and \cite[Example 11.42]{Ha95}.
\end{proof}

\begin{thm}\label{thmSBLD}
The movable cone of $\G(r,n)_1$ is given by
$$\Mov(\G(r,n)_1)=\begin{cases}
\left\langle H,H-rE\right\rangle &\mbox{ if } n=2r+1\\
\left\langle H,H-(r+1)E\right\rangle &\mbox{ if } n>2r+1
\end{cases}$$
and the divisors $E,H,H-E,\dots,H-(r+1)E$ give the walls of the stable base locus decomposition of $\Eff(\G(r,n)_1)$.
\end{thm}
\begin{proof}
In the same notation of Theorem \ref{main2} let $D$ be a $\mathbb{Q}$-divisor in $[E,H)$. Then $\textbf{B}(D)\subset E$. Furthermore, $D\cdot e < 0$ and since the curves as class $e$ cover $E$ we get that $E\subset \textbf{B}(D)$. Therefore, $\textbf{B}(D) = E$ for any $D\in [E,H)$.     

Now, let $D_1 = H + b_1 E$, $D_2 = H + b_2E$ be effective $\mathbb{Q}$-divisors in $\mathbb{G}(r,n)_1$ such that $b_2\leq b_1\leq 0$. Note that we can write
$$D_2 = D_1+(b_2-b_1)E$$
with $b_2-b_1 \leq 0$. Therefore
\stepcounter{thm}
\begin{equation}\label{sbl}
\textbf{B}(D_1)\subset \textbf{B}(D_2).
\end{equation}
Now, let $\widetilde{C} = mh-e$ be the strict transform of a degree $m$ rational normal curve in $\mathbb{G}(r,n)$ through the blown-up point $p\in \mathbb{G}(r,n)$, and let $D_{b} = H+bE$ be an effective $\mathbb{Q}$-divisor in $\mathbb{G}(r,n)_1$ with $-(m+1)\leq b < -m$. We have $-1\leq D_{b}\cdot \widetilde{C} <0$. Therefore, $\widetilde{R}_{m}\subset \textbf{B}(D_{b})$, where $\widetilde{R}_{m}$ is the strict transform of $R_m\subset \mathbb{G}(r,n)$ in $\mathbb{G}(r,n)_1$. On the other hand, the map induced by $H-(m+1)E$ is the restriction to $\mathbb{G}(r,n)_1$ of linear projection $\mathbb{P}^N\dasharrow\mathbb{P}^{N_m}$ with center $T^m_p\G(r,n)$, where we think $\G(r,n)\subset\mathbb{P}^N$ in its Pl\"ucker embedding . Hence the indeterminacy locus of the map induced by $H-(m+1)E$ is given by the strict transform of $T^m_p\G(r,n)\cap \G(r,n)$, and Lemma \ref{lemmaBS} yields that this strict transform is $\widetilde{R}_m$. Hence $\textbf{B}(D_{-m-1})\subset \widetilde{R}_m$ holds. Therefore (\ref{sbl}) yields that
$$\textbf{B}(D_{b}) = \widetilde{R}_m$$
for any $-(m+1)\leq b< -m$. 

This argument, Corollary \ref{dimBS} and (\ref{sbl}) give the claim on the movable cone and imply that the divisors $E,H,H-E,\dots,H-(r+1)E$ are the walls of the stable base locus decomposition of $\Eff(\G(r,n)_1)$. 
\end{proof}

\begin{Question}\label{q1}
Do the divisors $E,H,H-E,\dots,H-(r+1)E$ give the walls of the Mori chamber decomposition of $\G(r,n)_1$ as well? 
\end{Question}

Finally, we are ready to prove Theorem \ref{main2}.

\begin{proof}[{Proof of Theorem \ref{main2}}]
The case $r=0$ follows from the fact that $\Bl_{p_1,\dots,p_k}(\P^n)$ is toric if $k\leq n+1$,
and that $\Aut(\P^n,p_1,\dots,p_{n+2})=\{Id\}$. Now, assume $r\!\geq \!1$.

Regarding the spherical $\G(r,n)_k$'s, the case $k=1$ is in Proposition \ref{G(r,n)1 spherical}, the case $k=2$ follows 
from Corollary \ref{corollary k=2 c=1}, and the case $k\geq 3$ from Corollary \ref{corollary k>=3 c=1}. The effective cones were computed in Section \ref{sGbu}.
The statement on the movable cone of $\G(r,n)_1$, and the stable base locus decomposition of $\Eff(\G(r,n)_1)$ follow from Theorem \ref{thmSBLD}. The equality $\NE(\G(r,n)_1)=\left\langle e,h-e\right\rangle$ follows from Lemma \ref{conecurvesgeneral}. Then $\Nef(\G(r,n)_1) = \NE(\G(r,n)_1)^{\vee} = \left\langle H,H-E\right\rangle$. 

Finally, by \cite[Theorem 2.2]{BDPP13} we have that the cone of moving curves is dual to the effective cone, that is $\mov(\G(r,n)_1) = \Eff(\G(r,n)_1)^{\vee} = \left\langle h,(r+1)h-e\right\rangle$.
\end{proof}

\bibliographystyle{amsalpha}
\bibliography{Biblio}
\end{document}